\documentclass[nointlimits,11pt,oneside]{amsart}
\usepackage{amssymb,cases,enumitem, verbatim}
\usepackage{color}
\RequirePackage{doi}
\usepackage{hyperref}
\hypersetup{
	colorlinks=true,
	linkcolor=blue,
	citecolor=blue,
	filecolor=green,
	urlcolor=cyan,
	bookmarks=true
}

\usepackage[%
	a4paper,
	total={16cm,23cm},
	left=3cm, top=3cm,
	marginparsep=2pt]
{geometry}

\makeatletter
\newcommand{\opnorm}{\@ifstar\@opnorms\@opnorm}
\newcommand{\@opnorms}[1]{%
	\left|\mkern-1.5mu\left|\mkern-1.5mu\left|
	#1
	\right|\mkern-1.5mu\right|\mkern-1.5mu\right|
}
\newcommand{\@opnorm}[2][]{%
	\mathopen{#1|\mkern-1.5mu#1|\mkern-1.5mu#1|}
	#2
	\mathclose{#1|\mkern-1.5mu#1|\mkern-1.5mu#1|}
}
\makeatother
\makeatletter
\renewcommand*{\eqref}[1]{%
	\hyperref[{#1}]{\textup{\tagform@{\ref*{#1}}}}%
}
\makeatother

\setlist[enumerate,1]{label={\textup{(\arabic*)}}}

\theoremstyle{plain}
\newtheorem{theorem}{Theorem}[section]
\newtheorem{lemma}[theorem]{Lemma}
\newtheorem{corollary}[theorem]{Corollary}
\newtheorem{proposition}[theorem]{Proposition}
\theoremstyle{definition}
\newtheorem{remark}[theorem]{Remark}

\newtheorem{definition}[theorem]{Definition}

\numberwithin{equation}{section}

\def\L1loc{L^1_{\text{loc}}}

\begin{document}

\title{Wiener--Luxemburg amalgam spaces}
\author{Dalimil Pe{\v s}a}

\address{Dalimil Pe{\v s}a, Department of Mathematical Analysis, Faculty of Mathematics and
	Physics, Charles University, Sokolovsk\'a~83,
	186~75 Praha~8, Czech Republic}
\email{pesa@karlin.mff.cuni.cz}
\urladdr{0000-0001-6638-0913}

\subjclass[2010]{46E30, 46A16}
\keywords{Wiener--Luxemburg amalgam spaces, Wiener amalgam spaces, embedding theorems, associate spaces, rearrangement-invariant Banach function spaces, quasi-Banach function spaces, Hardy--Littlewood--P\'{o}lya principle}

\thanks{This research was supported by the grant P201-18-00580S of the Czech Science Foundation, by the Primus research programme PRIMUS/21/SCI/002 of Charles University, and by the Danube Region Grant no. 8X20043 TIFREFUS}

\begin{abstract}
In this paper we introduce the concept of Wiener--Luxemburg amalgam spaces which are a modification of the more classical Wiener amalgam spaces intended to address some of the shortcomings the latter face in the context of rearrangement-invariant Banach function spaces.

We introduce the Wiener--Luxemburg amalgam spaces and study their properties, including (but nor limited to) their normability, embeddings between them and their associate spaces. We also study amalgams of quasi-Banach function spaces and introduce a necessary generalisation of the concept of associate spaces. We then apply this general theory to resolve the question whether the Hardy--Littlewood--P\'{o}lya principle holds for all r.i.~quasi-Banach function spaces. Finally, we illustrate the asserted shortcomings of Wiener amalgam spaces by providing counterexamples to certain properties of Banach function spaces as well as rearrangement invariance.
\end{abstract}

\date{\today}

\maketitle

\makeatletter
   \providecommand\@dotsep{2}
\makeatother

\section{Introduction}

In this paper we introduce the concept of Wiener--Luxemburg amalgam spaces which are a modification of the more classical Wiener amalgam spaces. The principal idea of both kinds of amalgam spaces is to treat separately the local and global behaviour of a given function, in the sense that said function is required to be locally in one space and globally in a different space. The exact meaning of being locally and globally in a space varies in literature, depending on the desired generality and personal preference.

The classical Wiener amalgams approach this issue in a very general, albeit quite non-trivial, manner. They were, in their general form, first introduced by Feichtinger in \cite{Feichtinger83}, although the less general cases were studied earlier, see for example the paper \cite{Holland75} due to Holland, and some special cases date as far back as 1926 when the first example of such a space was introduced by Wiener in \cite{Wiener26}. The different versions of these spaces saw many applications in the last decades, great surveys of which have been conducted, concerning a somewhat restricted version, by Fournier and Stewart in \cite{FournierSteward85} and, concerning the more general versions, by Feichtinger in \cite{Feichtinger90} and \cite{Feichtinger92}. Probably the most famous example is the Tauberian theorem for the Fourier transform on the real line due to Wiener (see \cite{Wiener32} and~\cite{Wiener33}), other examples include the theory of Fourier multipliers (see \cite{EdwardsHewitt77}), several variation of the sampling theorem (see \cite{FeichtingerGrochenig92}), and the theory of product convolution--operators (see \cite{BusbySmith81}).

One unfortunate property of Wiener amalgams is that, even in the simplest and most natural cases, their construction does not preserve the properties of Banach function spaces, nor does it preserve rearrangement invariance (see Appendix~\ref{Appendix}). This approach is therefore unsuitable when one wishes to work in this context. But this often is the case, since there are many situations when the need arises naturally to prescribe separately the conditions on local and on global behaviour of a function. One such situation is the study of optimal Sobolev type embeddings over the entire Euclidean space in the context of rearrangement-invariant Banach function spaces as performed by Alberico, Cianchi, Pick and Slav{\' i}kov{\' a} in \cite{AlbericoCianchi2018}. A very natural example is the optimal target space, in the context of rearrangement-invariant Banach function spaces, for the limiting case of the classical Sobolev embeddings over the entire Euclidean space, which has been found by Vyb{\' i}ral in \cite{Vybiral07}. Another such situation arised during the study of generalised Lorentz--Zygmund spaces which led to the introduction of broken logarithmic functions to allow separate treatment of local and global properties of functions in this context. For further details and a comprehensive study of generalised Lorentz--Zygmund spaces we refer the reader to \cite{OpicPick99}. Further example of an area where this approach has been successfully employed is the theory of interpolation, where description of sums and intersections of spaces as de facto amalgams have proven useful, see \cite{Bathory18} and \cite{BennettRudnick80} for details.

This led us to develop the theory of Wiener--Luxemburg amalgam spaces, which aims to eliminate the above mentioned limitations and to provide a general framework for separate prescription of local and global conditions in the context of rearrangement-invariant Banach function spaces. The starting point is provided by the non-increasing rearrangement, which is the crucial element in the theory of said spaces and which naturally separates the local behaviour of a function from its global behaviour, at least in the sense of size. This allows us to define Wiener--Luxemburg amalgam spaces in a very easy and straightforward manner.

While the main part of this paper focuses on amalgams of rearrangement-invariant Banach function spaces, we will in the later sections also use the recent advances in the theory of quasi-Banach function spaces (see \cite{NekvindaPesa20}) and extend our theory into this context. While of independent interest, this will also allow us to view the theory of quasi-Banach function spaces from a new viewpoint and provide a negative answer to an important open question whether Hardy--Littlewood--P\'{o}lya principle holds for all rearrangement-invariant quasi-Banach function norms.

The paper is structured as follows. In Section~\ref{CHP} we present the basic theoretical background needed in order to build the theory in later sections.

Section~\ref{CHWLAS} is the main part of the paper where the theory of Wiener--Luxemburg amalgam spaces is developed in some detail for the more classical context of rearrangement-invariant Banach function spaces. We show that they too are rearrangement-invariant Banach function spaces, then we provide a characterisation of their associate spaces, a full characterisation of their embeddings, and put them in relation with the concepts of sum and intersection of Banach spaces. Furthermore we refine the well known classical result that it holds for all rearrangement-invariant Banach function spaces $A$ that
\begin{equation*}
L^1 \cap L^{\infty} \hookrightarrow A \hookrightarrow L^1 + L^{\infty}
\end{equation*}
by showing that $L^1$ is the locally weakest and globally strongest rearrangement-invariant Banach function space, while $L^{\infty}$ is, in the same setting, the locally strongest and globally weakest space. Needless to say, our definition of Wiener--Luxemburg amalgam spaces is general enough to cover all the spaces appearing in the applications outlined above.

We then in Section~\ref{CHIAS} switch to the more abstract context of rearrangement-invariant quasi-Banach function spaces and introduce an extension of the concept of associate spaces which we call the integrable associate spaces. This concept emerged naturally in the study of associate spaces of Wiener--Luxemburg amalgams in the context of quasi-Banach function spaces, so in this section we lay the groundwork for Section~\ref{CHWLASq}. However, we believe this topic to be interesting in its own right and the treatment we provide goes beyond what is strictly necessary for our later work.

Section~\ref{CHWLASq} then contains the extension of our theory to the context of rearrangement-invariant quasi-Banach function spaces. We show that, in this context, the Wiener--Luxemburg amalgam spaces are rearrangement-invariant quasi-Banach function spaces, we describe their integrable associate spaces (and, in the case when it is meaningful, their associate spaces), and provide a precise treatment of the above mentioned refinement of the result that
\begin{equation*}
L^1 \cap L^{\infty} \hookrightarrow A \hookrightarrow L^1 + L^{\infty}
\end{equation*}
which shows what properties of rearrangement-invariant Banach function norms are relevant for the individual embeddings. Last but not least, we show that the validity of the Hardy--Littlewood--P\'{o}lya principle is sufficient for one of the embeddings in question. Since there are spaces for which this embedding fails, we obtain a negative answer to the important open question whether the Hardy--Littlewood--P\'{o}lya principle holds for all r.i.~quasi-Banach function spaces.

Finally in Appendix~\ref{Appendix} we present some counterexamples which show that our claim that Wiener amalgams are in general neither Banach function spaces nor rearrangement-invariant is justified. This serves three distinct purposes: first, it fills a gap in literature; second, it provides some insight into the thought process behind our definition of Wiener--Luxemburg amalgam spaces; and third, it is an application of our theory, since we use Wiener--Luxemburg amalgams to show that Wiener amalgams are not rearrangement-invariant.

\section{Preliminaries}\label{CHP}
This section serves to establish the basic theoretical background upon which we will build our theory of Wiener--Luxemburg amalgam spaces. The definitions and notation is intended to be as standard as possible. The usual reference for most of this theory is \cite{BennettSharpley88}.

Throughout this paper we will denote by $(R, \mu)$, and occasionally by $(S, \nu)$, some arbitrary (totally) sigma-finite measure space. Given a $\mu$-measurable set $E \subseteq R$ we will denote its characteristic function by $\chi_E$. By $M(R, \mu)$ we will denote the set of all extended complex-valued $\mu$-measurable functions defined on $R$. As is customary, we will identify functions that coincide $\mu$-almost everywhere. We will further denote by $M_0(R, \mu)$ and $M_+(R, \mu)$ the subsets of $M(R, \mu)$ containing, respectively, the functions finite $\mu$-almost everywhere and the non-negative functions.

For brevity, we will abbreviate $\mu$-almost everywhere, $M(R, \mu)$, $M_0(R, \mu)$, and $M_+(R, \mu)$ to $\mu$-a.e., $M$, $M_0$, and $M_+$, respectively, when there is no risk of confusing the reader.

When $X, Y$ are two topological linear spaces, we will denote by $Y \hookrightarrow X$ that $Y \subseteq X$ and that the identity mapping $I : Y \rightarrow X$ is continuous.

As for some special cases, we will denote by $\lambda^n$ the classical $n$-dimensional Lebesgue measure, with the exception of the $1$-dimensional case in which we will simply write $\lambda$. We will further denote by $m$ the counting measure over $\mathbb{N}$. When $p \in (0, \infty]$ we will denote by $L^p$ the classical Lebesgue space (of functions in $M(R, \mu)$) defined by
\begin{equation*}
	L^p = \left \{ f \in M(R, \mu); \; \int_R \lvert f \rvert^p \: d\mu < \infty \right \}
\end{equation*}
equipped with the customary (quasi-)norm
\begin{equation*}
	\lVert f \rVert_p = \left ( \int_R \lvert f \rvert^p \: d\mu \right )^{\frac{1}{p} },
\end{equation*}
with the usual modifications when $p=\infty$. In the special case when $(R, \mu) = (\mathbb{N}, m)$ we will denote this space by $l^p$.

Note that in this paper we consider $0$ to be an element of $\mathbb{N}$.

\subsection{Non-increasing rearrangement}
We now present the concept of the non-increasing rearrangement of a function and state some of its properties that will be important later in the paper. We proceed in accordance with \cite[Chapter~2]{BennettSharpley88}.

We start by introducing the distribution function.

\begin{definition}	
	The distribution function $\mu_f$ of a function $f \in M$ is defined for $s \in [0, \infty)$ by
	\begin{equation*}
		\mu_f(s) = \mu(\{ t \in R; \; \lvert f(t) \rvert > s \}).
	\end{equation*}	
\end{definition}

The non-increasing rearrangement is then defined as the generalised inverse of the distribution function.

\begin{definition}
	The non-increasing rearrangement $f^*$ of a function $f \in M$ is defined for $t \in [0, \infty)$ by
	\begin{equation*}
		f^*(t) = \inf \{ s \in [0, \infty); \; \mu_f(s) \leq t \}.
	\end{equation*}
\end{definition}

For the basic properties of the distribution function and the non-increasing rearrangement, with proofs, see \cite[Chapter~2, Proposition~1.3]{BennettSharpley88} and \cite[Chapter~2, Proposition~1.7]{BennettSharpley88}, respectively. We consider those properties to be classical and well known and we will be using them without further explicit reference.

An important concept used in the paper is that of equimeasurability defined below.

\begin{definition} \label{DEM}
	We say that the functions $f \in M(R, \mu)$ and $g \in M(S, \nu)$ are equimeasurable if $\mu_f = \nu_g$.
\end{definition}

It is not hard to show that two functions are equimeasurable if and only if their non-increasing rearrangements coincide too.

A very important classical result is the Hardy--Littlewood inequality which we will use extensively in the paper. For proof, see for example \cite[Chapter~2, Theorem~2.2]{BennettSharpley88}.

\begin{theorem} \label{THLI}
	It holds for all $f, g \in M$ that
	\begin{equation*}
		\int_R \lvert fg \rvert \: d\mu \leq \int_0^{\infty} f^*g^* \: d\lambda.
	\end{equation*}
\end{theorem}

It follows directly from this result that it holds for every $f,g \in M$ that
\begin{equation*}
	\sup_{\substack{\tilde{g} \in M \\ \tilde{g}^* = g^*}} \int_R \lvert f \tilde{g} \rvert \: d\mu \leq \int_0^{\infty} f^*g^* \: d\lambda.
\end{equation*}
This motivates the definition of resonant measure spaces.

\begin{definition}
	A sigma-finite measure space $(R, \mu)$ is said to be resonant if it holds for all $f, g \in M(R, \mu)$ that
	\begin{equation*}
		\sup_{\substack{\tilde{g} \in M \\ \tilde{g}^* = g^*}} \int_R \lvert f \tilde{g} \rvert \: d\mu = \int_0^{\infty} f^* g^* \: d\lambda.
	\end{equation*}
\end{definition}

The property of being resonant is an important one. Luckily, there is a straightforward characterisation of resonant measure spaces. For proof and further details see \cite[Chapter~2, Theorem~2.7]{BennettSharpley88}.

\begin{theorem}
	A sigma-finite measure space is resonant if and only if it is either non-atomic or completely atomic with all atoms having equal measure.
\end{theorem}

\subsection{Norms and quasinorms} \label{SNQ}
In this subsection, and also in the following one, we provide the definitions for several classes of functionals we will study in the paper. All definitions should be standard or at least straightforward generalisations of standard ones.

The starting point shall be the class of norms.

\begin{definition} \label{DN}
	Let $X$ be a complex linear space. A functional $\lVert \cdot \rVert : X \rightarrow [0, \infty)$ will be called a norm if it satisfies the following conditions:
	\begin{enumerate}
		\item it is positively homogeneous, i.e.~$\forall a \in \mathbb{C} \; \forall x \in X : \lVert ax  \rVert = \lvert a \rvert \lVert x \rVert$,
		\item it satisfies $\lVert x \rVert = 0 \Leftrightarrow x = 0$  in $X$,
		\item it is subadditive, i.e.~$\forall x,y \in X : \lVert x+y \rVert \leq \lVert x \rVert + \lVert y \rVert$.
	\end{enumerate}
\end{definition}

Because the definition of a norm is sometimes too restrictive we will need a class of weaker functionals, namely quasinorms.

\begin{definition} \label{DQ}
	Let $X$ be a complex linear space. A functional $\lVert \cdot \rVert : X \rightarrow [0, \infty)$ will be called a quasinorm if it satisfies the following conditions:
	\begin{enumerate}
		\item it is positively homogeneous, i.e.~$\forall a \in \mathbb{C} \; \forall x \in X : \lVert ax  \rVert = \lvert a \rvert \lVert x \rVert$,
		\item it satisfies $\lVert x \rVert = 0 \Leftrightarrow x = 0$  in $X$,
		\item \label{DQ3} there is a constant $C\geq1$, called the modulus of concavity of $\lVert \cdot \rVert$, such that it is subadditive up to this constant, i.e.~$\forall x,y \in X \: : \: \lVert x+y \rVert \leq C(\lVert x \rVert + \lVert y \rVert)$.
	\end{enumerate}
\end{definition}

It is obvious that every norm is also a quasinorm with the modulus of concavity equal to $1$ and that every quasinorm with the modulus of concavity equal to $1$ is also a norm.

It is a well-known fact that every norm defines a metrizable topology on $X$ and that it is continuous with respect to that topology. This is not true for quasinorms, but this can be remedied thanks to the Aoki--Rolewicz theorem which we list below. Further details can be found for example in \cite{JohnsonLindenstraus03-25} or in \cite[Appendix~H]{BenyaminiLindenstrauss00}.

\begin{theorem}
	Let $\lVert \cdot \rVert_X$ be a quasinorm over the linear space $X$. Then there is a  quasinorm $\opnorm*{\cdot}_{X}$ over $X$ such that
	\begin{enumerate}
		\item there is a finite constant $C_0 > 0$ such that it holds for all $x \in X$ that
		\begin{equation*}
			C_0^{-1} \lVert x \rVert_X \leq \opnorm*{x}_{X} \leq C_0 \lVert x \rVert_X,
		\end{equation*}
		\item there is an $r \in (0, 1]$ such that it holds for all $x,y \in X$ that
		\begin{equation*}
			\opnorm*{x+y}_{X}^r \leq \opnorm*{x}_{X}^r + \opnorm*{y}_{X}^r.
		\end{equation*}
	\end{enumerate}
\end{theorem}

The direct consequence of this result is that every quasinorm defines a metrizable topology on $X$ and that the convergence in said topology is equivalent to the convergence with respect to the original quasinorm, in the sense that $x_n \rightarrow x$ in the induced topology if and only if $\lim_{n \rightarrow \infty} \lVert x_n - x \rVert = 0$.

Natural question to ask is when do different quasinorms define equivalent topologies. It is an easy exercise to show that the answer is the same as in the case of norms, that is that two quasinorms are topologically equivalent if and only if they are equivalent in the following sense.

\begin{definition}
	Let $\lVert \cdot \rVert_X$ and  $\opnorm*{\cdot}_{X}$ be quasinorms over the linear space $X$. We say that $\lVert \cdot \rVert_X$ and  $\opnorm*{\cdot}_{X}$ are equivalent if there is some $C_0 > 0$ such that it holds for all $x \in X$ that
	\begin{equation*}
		C_0^{-1} \lVert x \rVert_X \leq \opnorm*{\cdot}_{X} \leq C_0 \lVert x \rVert_X.
	\end{equation*}
\end{definition}

To conclude this part, we recall the concepts of sum and intersection of normed spaces.

\begin{definition}
	Let $X$ and $Y$ be normed linear spaces equipped with the norms $\lVert \cdot \rVert_X$ and $\lVert \cdot \rVert_Y$ respectively. Suppose that there is a Hausdorff topological linear space $Z$ into which $X$ and $Y$ are continuously embedded. We then define the spaces $X+Y$ and $X \cap Y$ as
	\begin{align*}
		X+Y &= \left \{z \in Z; \; \exists x \in X, \, \exists y \in Y : z = x +y  \right \}, \\
		X \cap Y &= \left \{z \in Z; \; z \in X, \, z \in Y  \right \},
	\end{align*}
	equipped with the norms
	\begin{align*}
		\lVert z \rVert_{X+Y} &= \inf \left \{\lVert x \rVert_X + \lVert y \rVert_Y; \; x \in X, \, y \in Y, \, x+y=z  \right \}, \\
		\lVert z \rVert_{X \cap Y} &= \max \left \{ \lVert z \rVert_X, \lVert z \rVert_Y \right \},
	\end{align*}
	respectively.
\end{definition}

The concepts presented above play a crucial role in the theory of interpolation. For further details, we refer the reader to \cite[Chapter~3]{BennettSharpley88}, where one can also find the following result (as \cite[Chapter~3, Theorem~1.3]{BennettSharpley88}).

\begin{theorem}
	Let $X$ and $Y$ be as above. Then $X+Y$ and $X \cap Y$, when equipped with their respective norms, are normed linear spaces. Furthermore, if $X$ and $Y$ are Banach spaces, then so are $X+Y$ and $X \cap Y$.
\end{theorem}

\subsection{Banach function norms and quasinorms}

We now turn our attention to the case in which we are interested the most, that is the case of norms and quasinorms acting on spaces of functions. The approach taken here is the same as in \cite[Chapter~1, Section~1]{BennettSharpley88}, which means that it differs, at least formally, from that in the previous part.

The major definitions are of course those of a Banach function norm and the corresponding Banach function space.

\begin{definition}
	Let $\lVert \cdot \rVert : M(R, \mu) \rightarrow [0, \infty]$ be a mapping satisfying $\lVert \, \lvert f \rvert \, \rVert = \lVert f \rVert$ for all $f \in M$. We say that $\lVert \cdot \rVert$ is a Banach function norm if its restriction to $M_+$ satisfies the following axioms:
	\begin{enumerate}[label=\textup{(P\arabic*)}, series=P]
		\item \label{P1} it is a norm, in the sense that it satisfies the following three conditions:
		\begin{enumerate}[ref=(\theenumii)]
			\item \label{P1a} it is positively homogeneous, i.e.\ $\forall a \in \mathbb{C} \; \forall f \in M_+ : \lVert a f \rVert = \lvert a \rvert \lVert f \rVert$,
			\item \label{P1b} it satisfies $\lVert f \rVert = 0 \Leftrightarrow f = 0$  $\mu$-a.e.,
			\item \label{P1c} it is subadditive, i.e.\ $\forall f,g \in M_+ \: : \: \lVert f+g \rVert \leq \lVert f \rVert + \lVert g \rVert$,
		\end{enumerate}
		\item \label{P2} it has the lattice property, i.e.\ if some $f, g \in M_+$ satisfy $f \leq g$ $\mu$-a.e., then also $\lVert f \rVert \leq \lVert g \rVert$,
		\item \label{P3} it has the Fatou property, i.e.\ if  some $f_n, f \in M_+$ satisfy $f_n \uparrow f$ $\mu$-a.e., then also $\lVert f_n \rVert \uparrow \lVert f \rVert $,
		\item \label{P4} $\lVert \chi_E \rVert < \infty$ for all $E \subseteq R$ satisfying $\mu(E) < \infty$,
		\item \label{P5} for every $E \subseteq R$ satisfying $\mu(E) < \infty$ there exists some finite constant $C_E$, dependent only on $E$, such that the inequality $ \int_E f \: d\mu \leq C_E \lVert f \rVert $ is true for all $f \in M_+$.
	\end{enumerate} 
\end{definition}

\begin{definition}
	Let $\lVert \cdot \rVert_X$ be a Banach function norm. We then define the corresponding Banach function space $X$ as the set
	\begin{equation*}
		X = \left \{ f \in M; \; \lVert f \rVert_X < \infty \right \}.
	\end{equation*}
\end{definition}

It is easy to see that a Banach function norm, when restricted to the space it defines, is indeed a norm in the sense of Definition~\ref{DN} and therefore Banach function spaces, when equipped with their defining norm, are normed linear spaces. Detailed study of these spaces can be found in \cite{BennettSharpley88}. 

Just as with general norms, the triangle inequality is sometimes too strong a condition to require. We therefore introduce the notions of quasi-Banach function norms and of the corresponding quasi-Banach function spaces.

\begin{definition}
	Let $\lVert \cdot \rVert : M(R, \mu) \rightarrow [0, \infty]$ be a mapping satisfying $\lVert \, \lvert f \rvert \, \rVert = \lVert f \rVert$ for all $f \in M$. We say that $\lVert \cdot \rVert$ is a quasi-Banach function norm if its restriction to $M_+$ satisfies the axioms \ref{P2}, \ref{P3} and \ref{P4} of Banach function norms together with a weaker version of axiom \ref{P1}, namely
	\begin{enumerate}[label=\textup{(Q\arabic*)}]
		\item \label{Q1} it is a quasinorm, in the sense that it satisfies the following three conditions:
		\begin{enumerate}[ref=(\theenumii)]
			\item \label{Q1a} it is positively homogeneous, i.e.\ $\forall a \in \mathbb{C} \; \forall f \in M_+ : \lVert af \rVert = \lvert a \rvert \lVert f \rVert$,
			\item \label{Q1b} it satisfies  $\lVert f \rVert = 0 \Leftrightarrow f = 0$ $\mu$-a.e.,
			\item \label{Q1c} there is a constant $C\geq 1$, called the modulus of concavity of $\lVert \cdot \rVert$, such that it is subadditive up to this constant, i.e.
			\begin{equation*}
				\forall f,g \in M_+ : \lVert f+g \rVert \leq C(\lVert f \rVert + \lVert g \rVert).
			\end{equation*}
		\end{enumerate}
	\end{enumerate}
\end{definition}

\begin{definition}
	Let $\lVert \cdot \rVert_X$ be a quasi-Banach function norm. We then define the corresponding quasi-Banach function space $X$ as the set
	\begin{equation*}
		X = \left \{ f \in M;  \; \lVert f \rVert_X < \infty \right \}.
	\end{equation*}
\end{definition}

As before, it is easy to see that a quasi-Banach function norm restricted to the space it defines is a quasinorm in the sense of Definition~\ref{DQ}. Let us now list here some of their important properties we will need later.

\begin{theorem} \label{TC}
	Let $\lVert \cdot \rVert_X$ be a quasi-Banach function norm and let $X$ be the corresponding quasi-Banach function space. Then $X$ is complete.
\end{theorem}

\begin{theorem} \label{TEQBFS}
	Let $\lVert \cdot \rVert_X$ and $\lVert \cdot \rVert_Y$ be quasi-Banach function norms and let $X$ and $Y$ be the corresponding quasi-Banach function spaces. If $X \subseteq Y$ then also $X \hookrightarrow Y$.
\end{theorem}

Both of these results have been known for a long time in the context of Banach function spaces but they have been only recently extended to quasi-Banach function spaces. Theorem~\ref{TC} has been first obtained by Caetano, Gogatishvili and Opic in \cite{CaetanoGogatishvili16} while Theorem~\ref{TEQBFS} has been proved by Nekvinda and the author in \cite{NekvindaPesa20}. The following result has also been obtained in \cite{NekvindaPesa20}:

\begin{theorem} \label{TP5}
	Let $\lVert \cdot \rVert_X$ be a quasi-Banach function norm and let $X$ be the corresponding quasi-Banach function space. Suppose that $E \subseteq R$ is a set such that $\mu(E) < \infty$ and that for every constant $K \in (0, \infty)$ there is a non-negative function $f \in X$ satisfying
	\begin{equation*}
	\int_E \lvert f \rvert \: d\mu > K \lVert f \rVert_X.
	\end{equation*}
	Then there is a non-negative function $f_E \in X$ such that
	\begin{equation*} \label{TP5.1}
	\int_E f_E \: d\mu = \infty.
	\end{equation*}
\end{theorem}

The last result concerning Banach function spaces we want to list at this point concerns the properties of the intersection of two Banach function spaces. The proof is an easy exercise.

\begin{proposition} \label{PIBFS}
	Let $X$ and $Y$ be two Banach function spaces. Then $X \cap Y$ is also a Banach function space.
\end{proposition} 

Let us now define an important property that a quasi-Banach function norm can have and that we will take a special interest in. Note that the class of quasi-Banach function norms contains that of Banach function norms so it is not necessary to provide separate definitions.

\begin{definition}
	Let $\lVert \cdot \rVert_X$ be a quasi-Banach function norm. We say that $\lVert \cdot \rVert_X$ is rearrangement-invariant, abbreviated r.i., if $\lVert f\rVert_X = \lVert g \rVert_X$ whenever $f, g \in M$ are equimeasurable (in the sense of Definition~\ref{DEM}).
	
	Furthermore, if the above condition holds, the corresponding space $X$ will be called rearrangement-invariant too.
\end{definition}

An important property of r.i.~quasi-Banach function spaces over $([0, \infty), \lambda)$ is that the dilation operator is bounded on those spaces, as stated in the following theorem. This is a classical result in the context of r.i.~Banach function spaces which has been recently extended to r.i.~quasi-Banach function spaces by by Nekvinda and the author in \cite{NekvindaPesa20} (for the classical version see for example \cite[Chapter~3, Proposition~5.11]{BennettSharpley88}).

\begin{definition} \label{DDO}
	Let $t \in (0, \infty)$. The dilation operator $D_t$ is defined on $M([0, \infty), \lambda)$ by the formula
	\begin{equation*}
	D_tf(s) = f(ts),
	\end{equation*}
	where $f \in M([0, \infty), \lambda)$, $s \in (0, \infty)$.
\end{definition}

\begin{theorem} \label{TDRIS}
	Let $X$ be an r.i.~quasi-Banach function space over $([0, \infty), \lambda)$ and let $t \in (0, \infty)$. Then $D_t: X \rightarrow X$ is a bounded operator.
\end{theorem}

Finally, we want to discuss here one property of some r.i.~quasi-Banach function norms that is often important in applications.

\begin{definition} \label{DHLP}
	Let $\lVert \cdot \rVert_X$ be an r.i.~quasi-Banach function norm. We say that the Hardy--Littlewood--P\'{o}lya principle holds for $\lVert \cdot \rVert_X$ if the estimate $\lVert f \rVert_X \leq \lVert g \rVert_X$ is true for any pair of functions $f, g \in M$ satisfying
	\begin{equation*}
	\int_0^{t} f^* \: d\lambda \leq \int_0^{t} g^* \: d\lambda 
	\end{equation*}
	for all $t \in (0, \infty)$.
\end{definition}

To put this property into the proper context we will need the following lemma:

\begin{lemma} \label{LHLP}
	Let $\lVert \cdot \rVert_X$ be an r.i.~quasi-Banach function norm and consider the following three statements:
	\begin{enumerate}
		\item \label{LHLPi}
		$\lVert \cdot \rVert_X$ is an r.i.~Banach function norm.
		\item \label{LHLPii}
		The Hardy--Littlewood--P\'{o}lya principle holds for $\lVert \cdot \rVert_X$.
		\item \label{LHLPiii}
		$\lVert \cdot \rVert_X$ satisfies \ref{P5}.
	\end{enumerate}
	Then \ref{LHLPi} implies \ref{LHLPii} which in turn implies \ref{LHLPiii}.	
\end{lemma}

\begin{proof}
	That \ref{LHLPi} implies \ref{LHLPii}, i.e.~that the Hardy--Littlewood--P\'{o}lya principle holds for all r.i.~Banach function spaces is a well known result, see for example \cite[Chapter~2, Theorem~4.6]{BennettSharpley88}.
	
	The remaining implication will be proved by contradiction, so we assume that $\lVert \cdot \rVert_X$ does not satisfy \ref{P5}. Then it follows from Theorem~\ref{TP5} that there is a set $E \subseteq R$ and a function $f \in M$ such that $\mu(E)<\infty$, $\lVert f \rVert_X < \infty$, and
	\begin{equation*}
		\int_E \lvert f \rvert \: d\mu = \infty.
	\end{equation*}
	Hence $\mu(E) > 0$ and the Hardy--Littlewood inequality (Theorem~\ref{THLI}) implies that also
	\begin{equation*}
		\int_0^{\mu(E)} f^* \: d\lambda = \infty.
	\end{equation*}
	Since $f^*$ is non-increasing, we finally obtain the equality
	\begin{equation*}
		\int_0^{t} f^* \: d\lambda = \infty
	\end{equation*}
	for all $t \in (0, \infty)$.
	
	Because we assume that the Hardy--Littlewood--P\'{o}lya principle holds for $\lVert \cdot \rVert_X$, we conclude that every function $g \in M$ satisfies $\lVert g \rVert_X < \infty$. Since this includes $g = \infty \chi_R$ and since $\mu(R) \geq \mu(E) > 0$, we obtain a contradiction with the property that quasi-Banach function spaces contain only functions that are finite almost everywhere (see \cite[Lemma~2.4]{MizutaNekvinda15} or \cite[Theorem~3.4]{NekvindaPesa20}).
\end{proof}

\begin{remark} \label{RHLP}
	None of the implications in Lemma~\ref{LHLP} can be reversed. In the first case, we can show this by considering the functional $\lVert \cdot \rVert_X = \lVert \cdot \rVert_{L^{p,q}}$, where $L^{p,q}$ is the Lorentz space, since for the choice of parameters $p \in (1, \infty)$, $q \in (0,1)$, those functionals are not even equivalent to norms (see \cite[Theorem~2.5.8]{CarroRaposo07} and the references therein) while the Hardy--Littlewood--P\'{o}lya principle still holds for them (this follows from the boundedness of the Hardy--Littlewood maximal operator, see \cite[Theorem~4.1]{CarroPick00} and the references therein). The second case is more interesting and the question whether the reverse implication holds has previously been open. We provide a negative answer in Corollary~\ref{CHLP}.
\end{remark}

\subsection{Associate space}

An important concept in the theory of Banach function spaces and their generalisations is that of an associate space. The detailed study of associate spaces of Banach function spaces can be found in \cite[Chapter~1, Sections~2, 3, and 4]{BennettSharpley88}.

We will approach the issue in a slightly more general way. The very definition of an associate space requires no assumptions on the functional defining the original space.

\begin{definition} \label{DAS}
	Let $\lVert \cdot \rVert_X: M \to [0, \infty]$ be some non-negative functional and put
	\begin{equation*}
		X = \{ f \in M; \; \lVert f \rVert_X < \infty \}.
	\end{equation*} 
	Then the functional $\lVert \cdot \rVert_{X'}$ defined for $f \in M$ by 
	\begin{equation}
	\lVert f \rVert_{X'} = \sup_{g \in X} \frac{1}{\lVert g \rVert_X} \int_R \lvert f g \rvert \: d\mu, \label{DAS1}
	\end{equation}
	where we interpret $\frac{0}{0} = 0$ and $\frac{a}{0} = \infty$ for any $a>0$, will be called the associate functional of $\lVert \cdot \rVert_X$ while the set
	\begin{equation*}
		X' = \left \{ f \in M; \; \lVert f \rVert_{X'} < \infty \right \}
	\end{equation*}
	will be called the associate space of $X$.
\end{definition}

As suggested by the notation, we will be interested mainly in the case when $\lVert \cdot \rVert_X$ is at least a quasinorm, but we wanted to emphasize that such an assumption is not necessary for the definition. In fact, it is not even required for the following result, which is the Hölder inequality for associate spaces.

\begin{theorem} \label{THAS}
	Let $\lVert \cdot \rVert_X: M \to [0, \infty]$ be some non-negative functional and denote by $\lVert \cdot \rVert_{X'}$ its associate functional. Then it holds for all $f \in M$ that
	\begin{equation*}
		\int_R \lvert f g \rvert \: d\mu \leq \lVert g \rVert_X \lVert f \rVert_{X'}
	\end{equation*}
	provided that we interpret $0 \cdot \infty = -\infty \cdot \infty = \infty$ on the right-hand side.
\end{theorem}

The convention at the end of the preceding theorem is necessary because the lack of assumptions on $\lVert \cdot \rVert_X$ means that we allow some pathological cases that need to be taken care of. To be more specific, $0 \cdot \infty = \infty$ is necessary because we allow $\lVert g \rVert_X = 0$ even for non-zero $g$ while $-\infty \cdot \infty = \infty$ is needed because the set $X$ can be empty, in which case $\lVert f \rVert_{X'} = \sup \emptyset = -\infty$.

The last result we will present in this generality is the following proposition concerning embeddings. Although the proof is an easy modification of that in \cite[Chapter~2, Proposition~2.10]{BennettSharpley88} we provide it to show that it truly does not require any assumptions on the original functional.

\begin{proposition} \label{PEASG}
	Let $\lVert \cdot \rVert_X: M \to [0, \infty]$ and $\lVert \cdot \rVert_Y: M \to [0, \infty]$ be two non-negative functionals satisfying that there is a constant $C>0$ such that it holds for all $f \in M$ that
	\begin{equation*}
		\lVert f \rVert_X \leq C \lVert f \rVert_Y.
	\end{equation*}
	Then the associate functionals $\lVert \cdot \rVert_{X'}$ and $\lVert \cdot \rVert_{Y'}$ satisfy, with the same constant $C$,
	\begin{equation*}
		\lVert f \rVert_{Y'} \leq C \lVert f \rVert_{X'}
	\end{equation*}
	for all $f \in M$.
\end{proposition}

\begin{proof}
	Our assumptions guarantee that $Y \subseteq X$ and therefore
	\begin{equation*}
		\begin{split}
			\lVert f \rVert_{Y'} &= \sup_{g \in Y} \frac{1}{\lVert g \rVert_Y} \int_R \lvert f g \rvert \: d\mu \\ 
			&\leq \sup_{g \in Y} \frac{C}{\lVert g \rVert_X} \int_R \lvert f g \rvert \: d\mu \\ 
			&\leq \sup_{g \in X} \frac{C}{\lVert g \rVert_X} \int_R \lvert f g \rvert \: d\mu \\
			&= C \lVert f \rVert_{X'}.
		\end{split}
	\end{equation*}
\end{proof}

Let us now turn our attention to the case when $\lVert \cdot \rVert_X$ is a quasi-Banach function norm. Note that in this case the definition of the associate functional does not change when one replaces the supremum in \eqref{DAS1} by one taken only over the unit sphere in $X$.

The following result, due to Gogatishvili and Soudsk{\'y} in \cite{GogatishviliSoudsky14}, shows that the conditions the original functional needs to satisfy in order for the associate functional to be a Banach function norm are quite mild. Specially, they are satisfied by any quasi-Banach function norm that satisfies the axiom \ref{P5}. This special case was observed earlier in \cite[Remark~2.3.(iii)]{EdmundsKerman00}.

\begin{theorem} \label{TFA}
	Let $\lVert \cdot \rVert_X : M \to [0, \infty]$ be a functional that satisfies the axioms \ref{P4} and \ref{P5} from the definition of Banach function norms and which also satisfies for all $f \in M$ that $\lVert f \rVert_X = \lVert \, \lvert f \rvert \, \rVert_X$. Then the functional $\lVert \cdot \rVert_{X'}$ is a Banach function norm. In addition, $\lVert \cdot \rVert_X$ is equivalent to a Banach function norm if and only if there is some constant $C \geq 1$ such that it holds for all $f \in M$ that
	\begin{equation} 
		\lVert f \rVert_{X''} \leq \lVert f \rVert_X \leq C \lVert f \rVert_{X''}, \label{TFA1}
	\end{equation}	
	where $\lVert \cdot \rVert_{X''}$ denotes the associate functional of $\lVert \cdot \rVert_{X'}$.
\end{theorem}

Additionally, if $\lVert \cdot \rVert_X$ is a Banach function norm then \eqref{TFA1} holds with $C = 1$. This is a classical result of Lorenz and Luxemburg, proof of which can be found for example in \cite[Chapter~1, Theorem~2.7]{BennettSharpley88}.

\begin{theorem} \label{TDAS}
	Let $\lVert \cdot \rVert_X$ be a Banach function norm, then $\lVert \cdot \rVert_X = \lVert \cdot \rVert_{X''}$ where $\lVert \cdot \rVert_{X''}$ is the associate functional of $\lVert \cdot \rVert_{X'}$.
\end{theorem}

Let us point out that even in the case when $\lVert \cdot \rVert_X$, satisfying the assumptions of Theorem~\ref{TFA}, is not equivalent to any Banach function norm we still have one interesting embedding, as formalised in the following statement. The proof is an easy exercise.

\begin{proposition} \label{PESSAS}
	Let $\lVert \cdot \rVert_X$ satisfy the assumptions of Theorem~\ref{TFA}. Then it holds for all $f \in M$ that
	\begin{equation*}
	\lVert f \rVert_{X''} \leq  \lVert f \rVert_X,
	\end{equation*}
	where $\lVert \cdot \rVert_{X''}$ denotes the associate functional of $\lVert \cdot \rVert_{X'}$.
\end{proposition}

We will also use in the paper the following version of Landau's resonance theorem. This result was first obtained in full generality in \cite{NekvindaPesa20}.

\begin{theorem} \label{LT}
	Let $\lVert \cdot \rVert_X$ be a quasi-Banach function norm, let $X$ be the corresponding quasi-Banach function space and let $\lVert \cdot \rVert_{X'}$ and $X'$, respectively, be the associate norm of  $\lVert \cdot \rVert_X$ and the corresponding associate space. Then a function $f \in M$ belongs to $X'$ if and only if it satisfies
	\begin{equation*}
	\int_R \lvert f g \rvert \: d\mu < \infty
	\end{equation*}
	for all $g \in X$.
\end{theorem}

To conclude this section, we observe that, provided the underlying measure space is resonant, the associate functional of an r.i.~quasi-Banach function norm can be expressed in terms of non-increasing rearrangement. The proof is the same as in \cite[Chapter~2, Proposition~4.2]{BennettSharpley88}.

\begin{proposition} \label{PAS}
	Let $\lVert \cdot \rVert_X$ be an r.i.~quasi-Banach function norm over a resonant measure space. Then its associate functional $\lVert \cdot \rVert_{X'}$ satisfies
	\begin{equation*}
		\lVert f \rVert_{X'} = \sup_{g \in X} \frac{1}{\lVert g \rVert_{X}} \int_0^{\infty} f^* g^* \: d\lambda.
	\end{equation*}
\end{proposition}

An obvious consequence of Proposition~\ref{PAS} is that an associate space of an r.i.~quasi-Banach function space (over a resonant measure space) is also rearrangement-invariant.

\section{Wiener--Luxemburg amalgam spaces} \label{CHWLAS}

Throughout this section we restrict ourselves to the case when $(R, \mu) = ([0, \infty), \lambda)$. This allows us to make the proofs more elegant and less technical as well as ensures that the underlying measure space is resonant. Note that this comes at no loss of generality, since any r.i.~Banach function space over an arbitrary resonant measure space can be represented by some r.i.~Banach function space over $([0, \infty), \lambda)$, as follows from the classical Luxemburg representation theorem (see for example \cite[Chapter~2, Theorem~4.10]{BennettSharpley88}), and any r.i.~Banach function space over $([0, \infty), \lambda)$ represents an r.i.~Banach function space over any resonant measure space (see for example \cite[Chapter~2, Theorem~4.9]{BennettSharpley88}).

\subsection{Wiener--Luxemburg quasinorms}

\begin{definition} \label{DefWL}
	Let $\lVert \cdot \rVert_A$ and $\lVert \cdot \rVert_B$ be r.i.~Banach function norms. We then define the Wiener--Luxemburg quasinorm $\lVert \cdot \rVert_{WL(A, B)}$, for $f \in M$, by
	\begin{equation}
	\lVert f \rVert_{WL(A, B)} = \lVert f^* \chi_{[0,1]} \rVert_A + \lVert f^* \chi_{(1, \infty)} \rVert_B \label{DefWLN}
	\end{equation}
	and the corresponding Wiener--Luxemburg amalgam space $WL(A, B)$ as
	\begin{equation*}
		WL(A, B) = \{f \in M; \; \lVert f \rVert_{WL(A, B)} < \infty \}.
	\end{equation*}
	
	Furthermore, we will call the first summand in \eqref{DefWLN} the local component of $\lVert \cdot \rVert_{WL(A, B)}$ while the second summand will be called the global component of $\lVert \cdot \rVert_{WL(A, B)}$.
\end{definition}

For the sake of brevity we will sometimes write just Wiener--Luxemburg amalgams instead of Wiener--Luxemburg amalgam spaces.

Let us at first note that this concept somewhat generalises the concept of the r.i.~Banach function spaces in the sense that every r.i.~Banach function space is, up to equivalence of the defining functionals, a Wiener--Luxemburg amalgam of itself.

\begin{remark} \label{RAS}
	Let $\lVert \cdot \rVert_A$ be an r.i.~Banach function norm. Then
	\begin{equation*}
		\lVert f \rVert_A \leq \lVert f \rVert_{WL(A, A)} \leq 2 \lVert f \rVert_A
	\end{equation*}
	for every $f \in M$.
	
	Consequently, it makes a good sense to talk about local and global components of arbitrary r.i.~Banach function norms. 
\end{remark}

The local component of an arbitrary r.i.~Banach function norm is worth separate attention. As shown in the following proposition, this component is itself an r.i.~Banach function norm and therefore any unpleasant behaviour of the Wiener--Luxemburg quasinorm must be caused by its global element. Second part of the proposition then illustrates the interesting fact that the global element of $L^{\infty}$ puts no additional condition on the size of a given function, in the sense that the spaces $WL(A,L^{\infty})$ consists of exactly those functions that are locally in $A$.

\begin{proposition} \label{PLC}
	Let $\lVert \cdot \rVert_A$ be an r.i.~Banach function norm. Then the functional
	\begin{equation*}
		f \mapsto \lVert f^* \chi_{[0,1]} \rVert_A
	\end{equation*}
	is also an r.i.~Banach function norm.
	
	Furthermore, there is a constant $C>0$ such that it holds for all $f \in M$ that
	\begin{equation} \label{PLCe}
	\lVert f^* \chi_{[0,1]} \rVert_A \leq \lVert f \rVert_{WL(A, L^{\infty})} \leq C \lVert f^* \chi_{[0,1]} \rVert_A.
	\end{equation}
\end{proposition}
\begin{proof}
	That the functional in question satisfies the axioms \ref{P2}, \ref{P3} and \ref{P4} as well as parts \ref{P1a} and \ref{P1b} of the axiom \ref{P1} is an easy consequence of the respective properties of $\lVert \cdot \rVert_A$ and the properties of non-increasing rearrangement. Furthermore, the rearrangement invariance is obvious.
	
	As for \ref{P5}, fix some set $E \subseteq [0, \infty)$ of finite measure. We may, without loss of generality, assume that $\lambda(E) > 1$, because otherwise the proof is similar but simpler. Then, by Hardy--Littlewood inequality (Theorem~\ref{THLI}), it holds for every $f \in M$ that
	\begin{equation*}
		\begin{split}
			\int_{E} f \: d\lambda &\leq \int_{0}^{\lambda(E)} f^* \: d\lambda = \int_0^{1} f^* \: d\lambda +
			\int_{1}^{\lambda(E)} f^* \: d\lambda \\
			&\leq \int_0^{1} f^* \: d\lambda + (\lambda(E) - 1) f^*(1) \leq \lambda(E) \int_0^{1} f^* \: d\lambda \leq
			\lambda(E) C_{[0,1]} \lVert f^*\chi_{[0,1]} \rVert_A,
		\end{split}
	\end{equation*}
	where $C_{[0,1]}$ is the constant from the property \ref{P5} of $\lVert \cdot \rVert_A$ for the set $[0,1]$.
	
	For the triangle inequality (part \ref{P1c} of axiom \ref{P1}) we employ the associate definition of~$\lVert \cdot \rVert_A$ (see Theorem~\ref{TDAS} and Proposition~\ref{PAS}) and the fact that $[0, \infty)$ is resonant to get for an arbitrary pair of functions $f, g \in M$ that
	\begin{equation*}
		\begin{split}
			\lVert (f+g)^* \chi_{[0,1]} \rVert_A &= \sup_{\lVert h \rVert_{A'} \leq 1} \int_0^{\infty} (f+g)^* \chi_{[0,1]}h^* \: d\lambda \\
			&= \sup_{\lVert h \rVert_{A'} \leq 1} \sup_{\tilde{h}^* = h^* \chi_{[0,1]}} \int_0^{\infty} (f+g) \tilde{h} \: d\lambda \\
			&\leq \sup_{\lVert h \rVert_{A'} \leq 1} \sup_{\tilde{h}^* = h^* \chi_{[0,1]}} \int_0^{\infty} f \tilde{h} \: d\lambda + \sup_{\lVert h \rVert_{A'} \leq 1} \sup_{\tilde{h}^* = h^* \chi_{[0,1]}} \int_0^{\infty} g \tilde{h} \: d\lambda \\
			&= \sup_{\lVert h \rVert_{A'} \leq 1} \int_0^{\infty} f^* \chi_{[0,1]}h^* \: d\lambda + \sup_{\lVert h \rVert_{A'} \leq 1} \int_0^{\infty} g^* \chi_{[0,1]}h^* \: d\lambda \\
			&= \lVert f^* \chi_{[0,1]} \rVert_A + \lVert g^* \chi_{[0,1]} \rVert_A.
		\end{split}
	\end{equation*}
	
	Thus we have shown that the functional in question is an r.i.~Banach function norm. It remains to show \eqref{PLCe}.
	
	The first inequality in \eqref{PLCe} is trivial. For the second estimate, it suffices to observe that
	\begin{equation*}
		\lVert f^* \chi_{(1, \infty)} \rVert_{L^{\infty}} = f^*(1) \leq \int_0^{1} f^* \: d\lambda \leq C_{[0,1]} \lVert f^*\chi_{[0,1]} \rVert_A, 
	\end{equation*}
	where $C_{[0,1]}$ is the constant from the property \ref{P5} of $\lVert \cdot \rVert_A$ for the set $[0,1]$.
\end{proof}

While the local component is an r.i.~Banach function norm, the global component is much less well behaved. Indeed, it is fairly easy to see that it cannot have the properties \ref{P1} and \ref{P5} (in \ref{P1} only part \ref{P1a} can possibly hold), because it cannot distinguish from zero any function that is supported on a set of measure less than one. Thus it makes no sense to consider it separately.

The following theorem shows that although Wiener--Luxemburg quasinorm needs not to be a norm, it satisfies all the remaining axioms of r.i.~Banach function norms. Note that this result is not redundant because in order to show that Wiener--Luxemburg amalgams are normable (see Corollary~\ref{CN}) we will use Theorem~\ref{TEQBFS} and Theorem~\ref{TFA} and it is thus necessary to establish first that Wiener--Luxemburg quasinorms are quasi-Banach function norms that satisfy the axiom \ref{P5}.

\begin{theorem} \label{TQN}
	The Wiener--Luxemburg quasinorms, as defined in Definition~\ref{DefWL}, are rearrangement-invariant quasi-Banach function norms and they also satisfy the axiom \ref{P5} from the definition of Banach function norms. Consequently, the corresponding Wiener--Luxemburg amalgam spaces are rearrangement-invariant quasi-Banach function spaces.
\end{theorem}

\begin{proof}
	The properties \ref{P2}, \ref{P3} and \ref{P4} as well as those from parts \ref{Q1a} and \ref{Q1b} of the axiom \ref{Q1} are easy consequences of the respective properties of $\lVert \cdot \rVert_A$ and $\lVert \cdot \rVert_B$ and the properties of non-increasing rearrangement. Furthermore, the rearrangement invariance is obvious.
	
	To show \ref{P5}, fix some set $E \subseteq [0, \infty)$ of finite measure. We may, without loss of generality, assume that $\lambda(E) > 1$, since otherwise the proof is similar but simpler. Then, by the Hardy--Littlewood inequality (Theorem~\ref{THLI}), it holds for every $f \in M_+$ that
	\begin{equation*}
		\begin{split}
			\int_{E} f \: d\lambda \leq \int_{0}^{\lambda(E)} f^* \: d\lambda &= \int_0^{1} f^* \: d\lambda + \int_1^{\lambda(E)} f^* \: d\lambda \\
			&\leq C_{[0,1]} \lVert f^* \chi_{[0,1]} \rVert_A + C_{(1, \lambda(E))} \lVert f^* \chi_{(1, \infty)} \rVert_B,
		\end{split}
	\end{equation*}
	where $C_{[0,1]}$ is the constant from the property \ref{P5} of $\lVert \cdot \rVert_A$ for the set $[0,1]$ and $C_{(1, \lambda(E))}$ is the constant from the same property of $\lVert \cdot \rVert_B$ for the set $(1, \lambda(E ))$.
	
	Finally, for the triangle inequality up to a multiplicative constant (part \ref{Q1c} of the axiom \ref{Q1}), consider the dilation operator $D_{\frac{1}{2}}$, as defined in Definition~\ref{DDO}, and use at first only the appropriate properties of non-increasing rearrangement and those of $\lVert \cdot \rVert_A$ and $\lVert \cdot \rVert_B$ to calculate
	\begin{align*}
		\lVert f+g \rVert_{WL(A, B)} &= \lVert (f+g)^* \chi_{[0,1]} \rVert_A + \lVert (f+g)^* \chi_{(1, \infty)} \rVert_B \\
		& \; \begin{aligned}
			\leq \lVert (D_{\frac{1}{2}} f^* + &D_{\frac{1}{2}} g^*) \chi_{[0,1]} \rVert_A \\ 
			&+ \lVert (D_{\frac{1}{2}} f^* + D_{\frac{1}{2}} g^*) \chi_{(1, \infty)} \rVert_B
		\end{aligned} \\
		& \; \begin{aligned}
			\leq \lVert D_{\frac{1}{2}} f^* &\chi_{[0,1]} \rVert_A + \lVert D_{\frac{1}{2}} g^* \chi_{[0,1]} \rVert_A \\
			&+ \lVert D_{\frac{1}{2}} f^* \chi_{(1, \infty)} \rVert_B + \lVert D_{\frac{1}{2}} g^* \chi_{(1, \infty)} \rVert_B,
		\end{aligned}
	\end{align*}
	which shows that it is sufficient to prove that there is some $C \in (0, \infty)$ such that
	\begin{equation*}
		\lVert D_{\frac{1}{2}} f^* \chi_{[0,1]} \rVert_A + \lVert D_{\frac{1}{2}} f^* \chi_{(1, \infty)} \rVert_B \leq C \lVert f \rVert_{WL(A, B)}
	\end{equation*}
	for all $f \in M_+$. Actually, it suffices to show
	\begin{equation} \label{TQN1}
	\lVert D_{\frac{1}{2}} f^* \chi_{(1, \infty)} \rVert_B \leq C \lVert f \rVert_{WL(A, B)},
	\end{equation}
	because $D_{\frac{1}{2}}$ is bounded on $A$ (by Theorem~\ref{TDRIS}), and thus
	\begin{equation*}
		\lVert D_{\frac{1}{2}} f^* \chi_{[0,1]} \rVert_A = \lVert D_{\frac{1}{2}}( f^* \chi_{[0,\frac{1}{2}]}) \rVert_A \leq \lVert D_{\frac{1}{2}} \rVert \lVert f^* \chi_{[0,\frac{1}{2}]} \rVert_A \leq \lVert D_{\frac{1}{2}} \rVert \lVert f^* \chi_{[0,1]} \rVert_A.
	\end{equation*}
	
	To show \eqref{TQN1}, fix some $f \in M_+$ and calculate
	\begin{equation*}
		\begin{split}
			\lVert D_{\frac{1}{2}} f^* \chi_{(1, \infty)} \rVert_B &= \lVert D_{\frac{1}{2}} (f^* \chi_{(\frac{1}{2}, \infty)}) \rVert_B \\
			&\leq \lVert D_{\frac{1}{2}} \rVert \lVert f^* \chi_{(\frac{1}{2}, \infty)} \rVert_B \\
			&\leq \lVert D_{\frac{1}{2}} \rVert (\lVert f^* \chi_{(1, \infty)} \rVert_B + \lVert f^* \chi_{(\frac{1}{2}, 1)} \rVert_B) \\
			&\leq \lVert D_{\frac{1}{2}} \rVert (\lVert f^* \chi_{(1, \infty)} \rVert_B + f^*(\tfrac{1}{2}) \lVert \chi_{(\frac{1}{2}, 1)} \rVert_B) \\
			&\leq \lVert D_{\frac{1}{2}} \rVert (\lVert f^* \chi_{(1, \infty)} \rVert_B + \lVert \chi_{(\frac{1}{2}, 1)} \rVert_B \lVert \chi_{(0, \frac{1}{2})} \rVert_A^{-1} \lVert f^*(\tfrac{1}{2}) \chi_{(0, \frac{1}{2})} \rVert_A   \\
			&\leq \lVert D_{\frac{1}{2}} \rVert (\lVert f^* \chi_{(1, \infty)} \rVert_B + \lVert \chi_{(\frac{1}{2}, 1)} \rVert_B \lVert \chi_{(0, \frac{1}{2})} \rVert_A^{-1} \lVert f^* \chi_{[0,1]} \rVert_A) \\
			&\leq \lVert D_{\frac{1}{2}} \rVert \max \{1, \lVert \chi_{(\frac{1}{2}, 1)} \rVert_B \lVert \chi_{(0, \frac{1}{2})} \rVert_A^{-1}\} \lVert f \rVert_{WL(A, B)}.
		\end{split}
	\end{equation*}
	This concludes the proof.
\end{proof}

\subsection{Associate spaces of Wiener--Luxemburg amalgams}

Let us now turn our attention to the associate spaces of Wiener--Luxemburg amalgams. The natural intuition here is that an associate space of an amalgam should be an amalgam of the respective associate spaces. This intuition turns out to be correct, as can be observed from the theorem presented below.

Note that while the conclusion of this theorem is natural, its proof is in fact quite involved. The difficulty stems from the fact that for a general function $f \in M$ the restriction of its non-increasing rearrangement $f^* \chi_{(1, \infty)}$ is not necessarily non-increasing and thus the quasinorm $\lVert f^* \chi_{(1, \infty)} \rVert_{WL(A, B)}$ actually depends not only on $\lVert \cdot \rVert_B$ but also on $\lVert \cdot \rVert_A$. This complicates things greatly and an entirely new method had to be developed to resolve the ensuing problems.

\begin{theorem} \label{TAS}
	Let $\lVert \cdot \rVert_A$ and $\lVert \cdot \rVert_B$ be r.i.~Banach function norms and let $\lVert \cdot \rVert_{A'}$ and $\lVert \cdot \rVert_{B'}$ be their respective associate norms. Then there is a constant $C>0$ such that the associate norm $\lVert \cdot \rVert_{(WL(A, B))'}$ of $\lVert \cdot \rVert_{WL(A, B)}$ satisfies
	\begin{equation}\label{TAS1}
		\lVert f \rVert_{(WL(A, B))'} \leq \lVert f \rVert_{WL(A', B')} \leq C \lVert f \rVert_{(WL(A, B))'}
	\end{equation}
	for every $f \in M$.
	
	Consequently, the corresponding associate space satisfies
	\begin{equation*}
		(WL(A, B))' = WL(A', B'),
	\end{equation*}
	up to equivalence of defining functionals.
\end{theorem}

\begin{proof}
	We begin by showing the first inequality in \eqref{TAS1}. To this end, fix some $f \in M$ and arbitrary $g \in M$ satisfying $\lVert g \rVert_{WL(A, B)} < \infty$. Then it follows from the Hölder inequality for associate spaces (Theorem~\ref{THAS}) that
	\begin{equation*}
		\begin{split}
			\int_0^{\infty} f^* g^* \: d\lambda &= \int_0^{\infty} f^*\chi_{[0,1]} g^* \: d\lambda + \int_0^{\infty} f^* \chi_{(1, \infty)} g^* \: d\lambda \\
			&\leq  \lVert f^* \chi_{[0,1]} \rVert_{A'} \lVert g^* \chi_{[0,1]} \rVert_A + \lVert f^* \chi_{(1,\infty)} \rVert_{B'} \lVert g^* \chi_{(1,\infty)} \rVert_B \\
			&\leq \max \{\lVert f^* \chi_{[0,1]} \rVert_{A'}, \: \lVert f^* \chi_{(1,\infty)} \rVert_{B'} \} \cdot \lVert g \rVert_{WL(A, B)} \\
			&\leq \lVert f \rVert_{WL(A', B')} \lVert g \rVert_{WL(A, B)}.
		\end{split}
	\end{equation*}
	The desired inequality now follows by dividing both sides of the inequality by $\lVert g \rVert_{WL(A, B)}$, taking the supremum over $WL(A,B)$ and using Proposition~\ref{PAS}.
	
	The second inequality in \eqref{TAS1} is more involved. We obtain it indirectly, showing first that $(WL(A,B))' \subseteq WL(A', B')$ and then using Theorem~\ref{TEQBFS}.
	
	Suppose that $f \notin WL(A', B')$. Then at least one of the following holds: $f^* \chi_{[0,1]} \notin A'$ or $f^* \chi_{(1,\infty)} \notin B'$. We treat these two cases separately and show that either is sufficient for $f \notin (WL(A,B))'$.
	
	If $f^* \chi_{[0,1]} \notin A'$ then we get by Theorem~\ref{LT} that there is a non-negative function $g \in A$ such that
	\begin{equation*}
		\int_0^{\infty} f^* \chi_{[0,1]} g \: d\lambda = \infty.
	\end{equation*}
	Now, $g^* \chi_{[0,1]} \in WL(A,B)$ because
	\begin{equation*}
		\lVert g^* \chi_{[0,1]} \rVert_{WL(A,B)} = \lVert g^* \chi_{[0,1]} \rVert_A \leq \lVert g^* \rVert_A = \lVert g \rVert_A < \infty
	\end{equation*}
	and we have by the Hardy--Littlewood inequality (Theorem~\ref{THLI}) the following estimate:
	\begin{equation*}
		\infty = \int_0^{\infty} f^* \chi_{[0,1]} g \: d\lambda \leq \int_0^{\infty} f^* g^* \chi_{[0,1]} \: d\lambda.
	\end{equation*}
	We have thus shown that $f \notin (WL(A,B))'$.
	
	Suppose now that $f^* \chi_{(1,\infty)} \notin B'$. We may assume that $f^*(1) < \infty$, because otherwise $f^* \chi_{[0,1]} = \infty \chi_{[0,1]} \notin A'$ and thus $f \notin (WL(A,B))'$ by the argument above. As in the previous case, we get by Theorem~\ref{LT} that there is some non-negative function $g \in B$ such that
	\begin{equation*}
		\int_0^{\infty} f^* \chi_{(1,\infty)} g \: d\lambda = \infty.
	\end{equation*}
	Now, it holds for all $t \in (0, \infty)$ that
	\begin{equation*}
		(f^* \chi_{(1,\infty)})^*(t) = f^*(t+1),
	\end{equation*}
	which, when combined with the Hardy--Littlewood inequality (Theorem~\ref{THLI}), yields
	\begin{equation*}
		\infty = \int_0^{\infty} f^* \chi_{(1,\infty)} g \: d\lambda \leq \int_0^{\infty} f^*(t+1) g^*(t) \: dt = \int_1^{\infty} f^*(t) g^*(t-1) \: dt.
	\end{equation*}
	If we now put
	\begin{equation*}
		\tilde{g}(t) = \begin{cases}
			0 & \text{for } t \in [0,1], \\
			g^*(t-1) & \text{for } t \in (1, \infty),
		\end{cases}
	\end{equation*}
	we immediately see that $\tilde{g}^* = g^*$ and thus we have found a function $\tilde{g} \in B$ that is zero on $[0, 1]$, non-increasing on $(1, \infty)$ and that satisfies
	\begin{equation*}
		\int_0^{\infty} f^* \chi_{(1,\infty)} \tilde{g} \: d\lambda = \infty.
	\end{equation*}
	Furthermore, we may estimate
	\begin{equation*}
		\int_1^{2} f^* \tilde{g} \: d\lambda \leq f^*(1) \int_1^2 \tilde{g} \: d\lambda \leq f^*(1)C_{[1,2]} \lVert \tilde{g} \rVert_{B} < \infty,
	\end{equation*}
	where $C_{[1,2]}$ is the constant from property \ref{P5} of $\lVert \cdot \rVert_B$. It follows that
	\begin{equation*}
		\int_2^{\infty} f^* \tilde{g} \: d\lambda = \infty,
	\end{equation*}
	because
	\begin{equation*}
		\infty = \int_0^{\infty} f^* \chi_{(1,\infty)} \tilde{g} \: d\lambda = \int_1^{2} f^* \tilde{g} \: d\lambda + \int_2^{\infty} f^* \tilde{g} \: d\lambda.
	\end{equation*}
	
	Finally, put $h = \tilde{g}(2) \chi_{[0,1]} + \min\{\tilde{g}, \tilde{g}(2)\}$. Note that $\tilde{g}(2) < \infty$ as follows from $\tilde{g} \in B$ and that $h$ is therefore a finite non-increasing function. Hence, we get that
	\begin{equation*}
		\lVert h \rVert_{WL(A,B)} = \tilde{g}(2) \lVert \chi_{(0,1)} \rVert_A + \lVert \min\{\tilde{g}, \tilde{g}(2)\} \rVert_{B} \leq \tilde{g}(2) \lVert \chi_{(0,1)} \rVert_A + \lVert \tilde{g} \rVert_{B} < \infty,
	\end{equation*}
	while by the arguments above we have
	\begin{equation*}
		\int_0^{\infty} f^* h^* \: d\lambda \geq \int_2^{\infty} f^* h^* \: d\lambda = \int_2^{\infty} f^* \tilde{g} \: d\lambda = \infty.
	\end{equation*}
	We therefore get that $f \notin (WL(A,B))'$. This covers the last case and establishes the desired inclusion $(WL(A,B))' \subseteq WL(A', B')$.
	
	Because we already know from Theorem~\ref{TQN} that $WL(A', B')$ is a quasi-Banach function space and from Theorem~\ref{TFA} that $(WL(A,B))'$ is a Banach function space, we may use Theorem~\ref{TEQBFS} to obtain $(WL(A,B))' \hookrightarrow WL(A', B')$, i.e.~that there is a constant $C>0$ such that it holds for all $f \in M$ that
	\begin{equation*}
		\lVert f \rVert_{WL(A', B')} \leq C \lVert f \rVert_{(WL(A, B))'},
	\end{equation*}
	which concludes the proof.
\end{proof}

As a corollary, we obtain normability of Wiener--Luxemburg amalgam spaces.

\begin{corollary} \label{CN}
	Let $\lVert \cdot \rVert_A$ and $\lVert \cdot \rVert_B$ be r.i.~Banach function norms. Then the Wiener--Luxemburg quasinorm $\lVert \cdot \rVert_{WL(A, B)}$ is equivalent to an r.i.~Banach function norm. Consequently, the Wiener--Luxemburg amalgam space $WL(A,B)$ is an r.i.~Banach function space.
\end{corollary}

\begin{proof}
	It follows from Theorem~\ref{TAS} that
	\begin{equation*}
		WL(A, B) = WL(A'', B'') = (WL(A',B'))',
	\end{equation*}
	where the space on the right-hand side is a Banach function space thanks to Theorem~\ref{TFA} and Theorem~\ref{TQN}.
\end{proof}

\subsection{Embeddings}

We now examine the embeddings of Wiener--Luxemburg amalgams. First we characterise the embeddings between two Wiener--Luxemburg amalgams, then between a Wiener--Luxemburg amalgam of two spaces and the sum or intersection of these spaces and finally we examine the case when the local or the global component is $L^1$ or $L^{\infty}$.

The first theorem provides the characterisation of embeddings among Wiener--Luxemburg amalgams.

\begin{theorem} \label{TEM}
	Let $\lVert \cdot \rVert_A$, $\lVert \cdot \rVert_B$, $\lVert \cdot \rVert_C$ and $\lVert \cdot \rVert_D$ be r.i.~Banach function norms. Then the following assertions are true:
	\begin{enumerate}
		\item The embedding $WL(A, C) \hookrightarrow WL(B,C)$ holds if and only if the local component of $\lVert \cdot \rVert_A$ is stronger that that of $\lVert \cdot \rVert_B$, in the sense that for every $f \in M$ the implication
		\begin{equation*}
			\lVert f^* \chi_{[0,1]} \rVert_A < \infty \Rightarrow \lVert f^* \chi_{[0,1]} \rVert_B < \infty
		\end{equation*} \label{TEMp1}
		holds.
		\item The embedding $WL(A, B) \hookrightarrow WL(A,C)$ holds if and only if the global component of $\lVert \cdot \rVert_B$ is stronger that that of $\lVert \cdot \rVert_C$, in the sense that for every $f \in M$ such that $f^*(1) < \infty$ the implication
		\begin{equation*}
			\lVert f^* \chi_{(1, \infty)} \rVert_B < \infty \Rightarrow \lVert f^* \chi_{(1, \infty)} \rVert_C < \infty
		\end{equation*}
		holds. \label{TEMp2}
		\item The embedding $WL(A,B) \hookrightarrow WL(C,D)$ holds if and only if the local component of $\lVert \cdot \rVert_A$ is stronger than that of $\lVert \cdot \rVert_C$ and the global component of $\lVert \cdot \rVert_B$ is stronger than that of $\lVert \cdot \rVert_D$. \label{TEMp3}
	\end{enumerate}
\end{theorem}

\begin{proof}
	In the first two cases the sufficiency follows directly from Theorem~\ref{TEQBFS} and Definition~\ref{DefWL}, only in the second case one has to realise that all $f \in WL(A, B)$ satisfy $f^*(1) < \infty$. The third case then follows, because we can use what we already proved to get
	\begin{equation*}
		WL(A,B) \hookrightarrow WL(A,D) \hookrightarrow WL(C,D).
	\end{equation*}
	
	The necessity in the case~\ref{TEMp1} can be shown in a following way. Fix some $f_0 \in M$ such that $\lVert f_0^* \chi_{[0,1]} \rVert_A < \infty$ but $\lVert f_0^* \chi_{[0,1]} \rVert_B = \infty$. Then $f = f_0^* \chi_{[0,1]}$ belongs to $WL(A,C)$, since 
	\begin{align*}
		\lVert f^* \chi_{[0,1]} \rVert_A &= \lVert f_0^* \chi_{[0,1]} \rVert_A < \infty, \\
		\lVert f^* \chi_{(1, \infty)} \rVert_C &= \lVert 0 \rVert_C = 0,
	\end{align*}
	but not to $WL(B,C)$, since 
	\begin{equation*}
		\lVert f^* \chi_{[0,1]} \rVert_B = \lVert f_0^* \chi_{[0,1]} \rVert_B = \infty.
	\end{equation*}
	
	As for the case~\ref{TEMp2}, fix some $f_0 \in M$ such that $f_0^*(1) < \infty$ and $\lVert f_0^* \chi_{(1, \infty)} \rVert_B < \infty$ while $\lVert f_0^* \chi_{(1, \infty)} \rVert_C = \infty$. Then $f = f_0^*(1) \chi_{[0,1]} + f_0^* \chi_{(1, \infty)}$ belongs to $WL(A,B)$, since 
	\begin{align*}
		\lVert f^* \chi_{[0,1]} \rVert_A &= \lVert f_0^*(1) \chi_{[0,1]} \rVert_A = f_0^*(1) \lVert \chi_{[0,1]} \rVert_A < \infty, \\
		\lVert f^* \chi_{(1, \infty)} \rVert_B &= \lVert f^*_0 \chi_{(1, \infty)} \rVert_B < \infty,
	\end{align*}
	but not to $WL(A,C)$, since
	\begin{equation*} 
		\lVert f^* \chi_{(1, \infty)} \rVert_C = \lVert f^*_0 \chi_{(1, \infty)} \rVert_C = \infty.
	\end{equation*}
	
	Finally for the case~\ref{TEMp3} one needs to combine the steps presented above. To be precise, if the condition on local components is violated one finds the counterexample as in the case~\ref{TEMp1} while in the case when the condition on the global components gets violated one finds it as in the case~\ref{TEMp2}.
\end{proof}

To provide an example we turn to the classical Lebesgue spaces. It is well known that Lebesgue spaces over $[0,\infty)$ are not ordered, but it is easy to show that their local and global component are, as is formalised in the following remark. 

\begin{remark} \label{RELp}
	Let $p,q \in (0,\infty]$. Then it holds that
	\begin{enumerate}
		\item the local component of $\lVert \cdot \rVert_{L^p}$ is stronger than that of $\lVert \cdot \rVert_{L^q}$ if and only if $p \geq q$,
		\item the global component of $\lVert \cdot \rVert_{L^p}$ is stronger than that of $\lVert \cdot \rVert_{L^q}$ if and only if $p \leq q$.
	\end{enumerate}
\end{remark}

As a second example we present a similar statement about Lorentz spaces. The proof is easy and uses only standard techniques.

\begin{remark} \label{RELpq}
	Let $p_1, p_2, q_1, q_2 \in (0,\infty]$ such that $p_1 \neq p_2$ and let $\lVert \cdot \rVert_{L^{p_1,q_1}}$ and $\lVert \cdot \rVert_{L^{p_2,q_2}}$ be the corresponding Lorentz functionals. Then it holds that
	\begin{enumerate}
		\item the local component of $\lVert \cdot \rVert_{L^{p_1,q_1}}$ is stronger than that of $\lVert \cdot \rVert_{L^{p_2,q_2}}$ if and only if $p_1 > p_2$,
		\item the global component of $\lVert \cdot \rVert_{L^{p_1,q_1}}$ is stronger than that of $\lVert \cdot \rVert_{L^{p_2,q_2}}$ if and only if $p_1 < p_2$.
	\end{enumerate}
\end{remark}

A third example might be found among Orlicz spaces. The following remark contains sufficient conditions for the ordering of their respective local and global components. It is again easy to prove, using only the already well-known methods which have been originally developed for characterising the embeddings between Orlicz spaces and which can be found for example in \cite[Theorem~4.17.1]{FucikKufner13}. We also refer the reader to \cite[Chapter~4]{FucikKufner13} for an extensive treatment of Orlicz spaces.

\begin{remark}
	Let $\Phi_1$ and $\Phi_2$ be two Young functions and let $\lVert \cdot \rVert_{\Phi_1}$ and $\lVert \cdot \rVert_{\Phi_2}$ be the corresponding Orlicz norms. Then
	\begin{enumerate}
		\item if there are some constants $c, T \in (0, \infty)$ such that
		\begin{equation*}
			\Phi_2(t) \leq \Phi_1(ct)
		\end{equation*}
		for all $t \in [T, \infty)$ then the local component of $\lVert \cdot \rVert_{\Phi_1}$ is stronger than that of $\lVert \cdot \rVert_{\Phi_2}$,
		\item if there are some constants $c, T \in (0, \infty)$ such that
		\begin{equation*}
			\Phi_2(t) \leq \Phi_1(ct)
		\end{equation*}
		for all $t \in [0, T]$ then the global component of $\lVert \cdot \rVert_{\Phi_1}$ is stronger than that of $\lVert \cdot \rVert_{\Phi_2}$.
	\end{enumerate}
\end{remark}

We now put $WL(A,B)$ in relation with the sum and intersection of $A$ and $B$. We first show that $WL(A,B)$ is always sandwiched between them.

\begin{theorem}\label{TRS}
	Let $\lVert \cdot \rVert_A$ and $\lVert \cdot \rVert_B$ be r.i.~Banach function norms. Then
	\begin{equation*}
		A \cap B \hookrightarrow WL(A,B) \hookrightarrow A + B.
	\end{equation*}
\end{theorem}

\begin{proof}
	Fix some $f \in M$. Then
	\begin{equation*}
		\lVert f \rVert_{WL(A, B)} = \lVert f^* \chi_{[0,1]} \rVert_A + \lVert f^* \chi_{(1, \infty)} \rVert_B \leq \lVert f \rVert_A + \lVert f \rVert_B \leq 2 \lVert f \rVert_{A \cap B}
	\end{equation*}
	which establishes the first embedding. 
	
	As for the second embedding, note that we may consider $f$ to be non-negative, since it is easy to show that it holds for every $f \in M$ that $\lVert f \rVert_{A+B} = \lVert \, \lvert f \rvert \, \rVert_{A+B}$.
	
	Consider now functions $g$ and $h$ defined by
	\begin{align*}
		g &= \max \{ f - f^*(1), 0\}, \\
		h &= \min \{ f, f^*(1)\}.
	\end{align*}
	Then $f = g + h$ and thus
	\begin{equation*}
		\begin{split}
			\lVert f \rVert_{A+B} &\leq \lVert g \rVert_A + \lVert h \rVert_B = \lVert g^* \rVert_A + \lVert h^* \rVert_B
		\end{split} 
	\end{equation*}
	thanks to rearrangement invariance of both $\lVert \cdot \rVert_A$ and $\lVert \cdot \rVert_B$. Furthermore, thanks to $f$ beeing non-negative, it is an exercise to verify that
	\begin{align*}
		g^* &= (f^* - f^*(1)) \chi_{[0,1]}, \\
		h^* &= f^*(1) \chi_{[0,1]} + f^* \chi_{(1, \infty)},
	\end{align*}
	and therefore
	\begin{equation*}
		\begin{split}
			\lVert f \rVert_{A+B} &\leq \lVert f^* \chi_{[0,1]} \rVert_A + \lVert f^*(1) \chi_{[0,1]} \rVert_B + \lVert f^* \chi_{(1, \infty)} \rVert_B \\
			&\leq \lVert f \rVert_{WL(A, B)} + \lVert \chi_{[0,1]} \rVert_B \int_0^1 f^* \: d\lambda \\
			&\leq (1 + C_{[0,1]}  \lVert \chi_{[0,1]} \rVert_B ) \lVert f \rVert_{WL(A, B)},
		\end{split} 
	\end{equation*}
	where $C_{[0,1]}$ is the constant from the property \ref{P5} of $\lVert \cdot \rVert_A$ for the set $[0,1]$. This establishes the second embedding.
\end{proof}

Moreover, in the case when we have proper relations between the respective components of $A$ and $B$ we can describe their sum and intersection in terms of Wiener--Luxemburg amalgams, at least in the set theoretical sense.

\begin{corollary}\label{CIS}
	Let $\lVert \cdot \rVert_A$ and $\lVert \cdot \rVert_B$ be r.i.~Banach function norms. Suppose that the local component of $\lVert \cdot \rVert_A$ is stronger than that of $\lVert \cdot \rVert_B$ while the global component of $\lVert \cdot \rVert_B$ is stronger than that of $\lVert \cdot \rVert_A$. Then
	\begin{equation*}
		A \cap B = WL(A,B)
	\end{equation*}
	up to equivalence of quasinorms, while
	\begin{equation} \label{CIS1}
		A + B = WL(B,A)
	\end{equation}
	as sets.
\end{corollary}

\begin{proof}
	Thanks to Proposition~\ref{PIBFS}, Theorem~\ref{TEQBFS} and Theorem~\ref{TRS} it suffices to prove that $WL(A,B) \subseteq A \cap B$ and $A + B \subseteq WL(B,A)$. But this is provided by Theorem~\ref{TEM} and Remark~\ref{RAS}, which, thanks to our assumptions, yield
	\begin{align*}
		WL(A,B) &\hookrightarrow A, \\
		WL(A,B) &\hookrightarrow B, \\
		A &\hookrightarrow WL(B,A), \\
		B &\hookrightarrow WL(B,A).
	\end{align*}
	When combined with the fact that $WL(B,A)$ is a linear set, this is sufficient for the inclusions.
\end{proof}

The reason for the equality \eqref{CIS1} holding only in the set theoretical sense is of course the fact that $A+B$ is not necessarily a Banach function space. This also motivates the following observation.

\begin{corollary}
	Let $\lVert \cdot \rVert_A$ and $\lVert \cdot \rVert_B$ be r.i.~Banach function norms. Suppose that the local component of $\lVert \cdot \rVert_A$ is stronger than that of $\lVert \cdot \rVert_B$ while the global component of $\lVert \cdot \rVert_B$ is stronger than that of $\lVert \cdot \rVert_A$. Then there is an r.i.~Banach function norm $\lVert \cdot \rVert_X$ such that there is a constant $C >0$ satisfying
	\begin{equation*}
		\lVert f \rVert_{A+B} \leq C \lVert f \rVert_X
	\end{equation*}
	for all $f \in A+B$ and such that the corresponding r.i.~Banach function space $X$ satisfies
	\begin{equation*}
		X = A + B
	\end{equation*}
	as a set.
\end{corollary}

In the next theorem, we show that the classical Lebesgue space $L^1$ has the weakest local component, as well as the strongest global component, among all r.i.~Banach function spaces, while $L^{\infty}$ has, in the same context, the strongest local component as well as the weakest global component.

\begin{theorem} \label{TLL}
	Let $\lVert \cdot \rVert_A$ and $\lVert \cdot \rVert_B$ be r.i.~Banach function norms and let $A$ and $B$ be the corresponding r.i.~Banach function spaces. Then
	\begin{enumerate}
		\item  $WL(L^{\infty}, B) \hookrightarrow WL(A,B)$, \label{TLLp1}
		\item  $WL(A, L^1) \hookrightarrow WL(A,B)$, \label{TLLp2}
		\item  $WL(A, B) \hookrightarrow WL(L^1, B)$, \label{TLLp3}
		\item  $WL(A, B) \hookrightarrow WL(A, L^{\infty})$. \label{TLLp4}
	\end{enumerate}
\end{theorem}

\begin{proof}
	Fix $f \in M$. The first embedding follows from the estimate
	\begin{equation*}
		\lVert f^* \chi_{[0,1]} \rVert_A \leq \lVert f^* \chi_{[0,1]} \rVert_{L^{\infty}} \lVert \chi_{[0,1]} \rVert_A 
	\end{equation*}
	and part~\ref{TEMp1} of Theorem~\ref{TEM}.
	
	The third embedding also uses part~\ref{TEMp1} of Theorem~\ref{TEM} but this time paired with the estimate
	\begin{equation*}
		\lVert f^* \chi_{[0,1]} \rVert_{L^1} = \int_0^1 f^* \chi_{[0,1]} \: d\lambda \leq C_{[0,1]} \lVert f^* \chi_{[0,1]} \rVert_A,
	\end{equation*}
	where $C_{[0,1]}$ is the constant from property \ref{P5} of $\lVert \cdot \rVert_A$.
	
	The fourth embedding is a trivial consequence of Proposition~\ref{PLC}, specifically of \eqref{PLCe}.
	
	The second embedding is most involved. Denote by $\lVert \cdot \rVert_{B'}$ the associate norm of $\lVert \cdot \rVert_B$ and by $B'$ the associate space of $B$. Then we know from part~\ref{TLLp4}, which has already been proved, and Remark~\ref{RAS} that $B' \hookrightarrow WL(B',L^{\infty})$. Thus it follows from Theorem~\ref{TAS} and Proposition~\ref{PEASG} that
	\begin{equation*}
		WL(B, L^1) = (WL(B', L^{\infty}))' \hookrightarrow B'' = B.
	\end{equation*}
	Hence, it follows from part~\ref{TEMp2} of Theorem~\ref{TEM} that the global component of $\lVert \cdot \rVert_{L^1}$ is stronger than that of $\lVert \cdot \rVert_B$ which, by the same theorem, implies the desired embedding.
\end{proof}

As a corollary, we obtain the following well-known classical result, for which we thus provide an alternative proof.

\begin{corollary}
	Let $A$ be an r.i.~Banach function space. Then
	\begin{equation*}
		L^1 \cap L^{\infty} \hookrightarrow A \hookrightarrow L^1 + L^{\infty}.
	\end{equation*}
\end{corollary}

\begin{proof}
	The assertion is a direct consequence of Remark~\ref{RELp}, Corollary~\ref{CIS}, Theorem~\ref{TLL}, Theorem~\ref{TEQBFS} and the fact that $L^1 + L^{\infty}$ is an r.i.~Banach function space.
\end{proof}

The final result of this section is a more precise version of Proposition~\ref{PEASG}.

\begin{proposition} \label{PEAS}
	Let $\lVert \cdot \rVert_A$ and $\lVert \cdot \rVert_B$ be r.i.~Banach function norms and denote by $\lVert \cdot \rVert_ {A'}$ and $\lVert \cdot \rVert_{B'}$ the respective associate norms. Suppose that the local component of $\lVert \cdot \rVert_A$ is stronger than that of $\lVert \cdot \rVert_B$. Then the local component of $\lVert \cdot \rVert_{B'}$ is stronger than that of $\lVert \cdot \rVert_{A'}$.
	
	Similarly, if the global component of $\lVert \cdot \rVert_A$ is stronger than that of $\lVert \cdot \rVert_B$, then the global component of $\lVert \cdot \rVert_{B'}$ is stronger than that of $\lVert \cdot \rVert_{A'}$
\end{proposition}

\begin{proof}
	By our assumption and part~\ref{TEMp1} of~Theorem~\ref{TEM} we get that
	\begin{equation*}
		WL(A, L^{\infty}) \hookrightarrow WL(B, L^{\infty}).
	\end{equation*}
	Consequently, it follows from Theorem~\ref{TAS} and Proposition~\ref{PEASG} that
	\begin{equation*}
		WL(B', L^1) = (WL(B, L^{\infty}))' \hookrightarrow (WL(A, L^{\infty}))' = WL(A', L^1),
	\end{equation*}
	that is, the local component of $\lVert \cdot \rVert_{B'}$ is stronger than that of $\lVert \cdot \rVert_ {A'}$.
	
	The second claim can be proved in similar manner, only using $WL(L^1, A)$ and $WL(L^1, B)$ instead of $WL(A, L^{\infty})$ and $WL(B, L^{\infty})$.
\end{proof}

\section{Integrable associate spaces} \label{CHIAS}
In this section we introduce a generalisation of the concept of associate spaces (see Definition~\ref{DAS}) the need of which arose naturally during the study of associate spaces of Wiener--Luxemburg amalgams of r.i.~quasi-Banach function spaces. We then study some of its properties, mainly those directly needed for our purposes.

Unlike in the other parts of the paper, we will now work in a more abstract setting and assume only that $(R, \mu)$ is a resonant measure space.

Our terminology and notation in this section is inspired by the down associate norm which is derived from the concept of down norms. To those interested in this topic we suggest the papers \cite{EdmundsKerman00} and \cite{Sinnamon01}.

\subsection{Basic properties}

Our definition of integrable associate spaces is rather indirect. We proceed by first introducing a certain subspace of an arbitrary r.i.~quasi-Banach function space and then defining the integrable associate space as an associate space of this subspace.

\begin{definition}
	Let $\lVert \cdot \rVert_X$ be an r.i.~quasi-Banach function norm and let $X$ be the corresponding r.i.~quasi-Banach function space. Then the functional $\lVert \cdot \rVert_{X_i}$, defined for $f \in M$ by
	\begin{equation*}
		\lVert f \rVert_{X_i} = \max \{ \lVert f \rVert_X, \, \lVert f \rVert_{WL(L^1, L^{\infty})}  \},
	\end{equation*}
	will be called the integrable norm of $\lVert \cdot \rVert_X$, while the space
	\begin{equation*}
	X_i = \left \{ f \in M; \; \lVert f \rVert_{X_i} < \infty \right \}
	\end{equation*}
	will be called integrable subspace of $X$.
\end{definition}

\begin{theorem} \label{TISS}
	Let $\lVert \cdot \rVert_X$ be an r.i.~quasi-Banach function norm and let $X$ be the corresponding r.i.~quasi-Banach function space. Denote by $C$ the modulus of concavity of $\lVert \cdot \rVert_X$. Then the functional $\lVert \cdot \rVert_{X_i}$ is an r.i.~quasi-Banach function norm that has the property \ref{P5}. Furthermore, if $C = 1$, then $\lVert \cdot \rVert_{X_i}$ is an r.i.~Banach function norm.
	
	Moreover, the space $X_i$ satisfies
	\begin{equation} \label{TISS1}
		X_i = \left \{ f \in X; \; \int_0^{1} f^* \: d\lambda < \infty \right \}
	\end{equation}
	and thus $X_i \hookrightarrow X$. On the other hand, if $\lVert \cdot \rVert_X$ has the property \ref{P5} then $X_i = X$ up to equivalence of quasinorms.
\end{theorem}

\begin{proof}
	Given that we know from Proposition~\ref{PLC} that $\lVert \cdot \rVert_{WL(L^1, L^{\infty})}$ is an r.i.~Banach function norm, it is an easy exercise to show that the first part of the theorem holds, i.e. that $\lVert \cdot \rVert_{X_i}$ has the asserted properties. The characterisation of $X_i$ also follows from Proposition~\ref{PLC} while the embedding $X_i \hookrightarrow X$ can then be obtained via Theorem~\ref{TEQBFS} (or directly from the definition of $\lVert \cdot \rVert_{X_i}$).
	
	The last part is slightly more involved. Note that we cannot simply use the Luxemburg representation theorem to show that every $f \in X$ belongs to $X_i$, since $\lVert \cdot \rVert_X$ needs not to be a norm. We thus proceed as follows.
	
	We will assume that $\mu(R) \geq 1$; the remaining case is easier. If $\lVert \cdot \rVert_X$ has the property \ref{P5}, then its associate norm $\lVert \cdot \rVert_{X'}$ is an r.i.~Banach function norm and we can use it, as well as our assumption that the underlying measure space is resonant, to obtain an estimate on 
	\begin{equation*}
		\int_0^{1} f^* \: d\lambda.
	\end{equation*}
	To this end, we use the $\sigma$-finiteness of $(R, \mu)$ to find some $a \in [1, \infty)$ such that there is a set $E$ with $\mu(E) = a$. Then it follows from the Hölder inequality for associate spaces (Theorem~\ref{THAS}) that
	\begin{equation} \label{TISS2}
		\int_0^{1} f^* \: d\lambda \leq \int_0^{a} f^* \: d\lambda = \sup_{\substack{E \subseteq R \\ \mu(E) = a}} \int_E \lvert f \rvert \: d\mu \leq \sup_{\substack{E \subseteq R \\ \mu(E) = a}} \lVert \chi_E \rVert_{X'} \lVert f \rVert_X = C_a \lVert f \rVert_X
	\end{equation}
	where $C_a$ is the norm $\lVert \chi_E \rVert_{X'}$ for any set $E \subseteq R$ such that $\mu(E) = a$. Hence, it follows from \eqref{TISS1} that the sets $X$ and $X_i$ coincide. The equivalence of quasinorms then follows from Theorem~\ref{TEQBFS} (or directly from \eqref{TISS2}).
\end{proof}

\begin{definition}
	Let $\lVert \cdot \rVert_X$ be an r.i.~quasi-Banach function norm, let $X$ be the corresponding r.i.~quasi-Banach function space and let $\lVert \cdot \rVert_{X_i}$ and $X_i$, respectively, be the corresponding integrable norm and integrable subspace. Then the associate norm $\lVert \cdot \rVert_{X_i'}$ of $\lVert \cdot \rVert_{X_i}$ will also be called the integrable associate norm of $\lVert \cdot \rVert_X$, while the associate space $X_i'$ of $X_i$ will also be called the integrable associate space of $X$.
\end{definition}

The next two results describe the properties of the integrable associate spaces and their relation to the associate spaces.

\begin{corollary} \label{CIAS}
	Let $\lVert \cdot \rVert_X$ be an r.i.~quasi-Banach function norm, let $X$ be the corresponding r.i.~quasi-Banach function space and let $\lVert \cdot \rVert_{X_i'}$ and $X_i'$, respectively, be the corresponding integrable associate norm and integrable associate space. Then $\lVert \cdot \rVert_{X_i'}$ is an r.i.~Banach function norm and $X_i'$ is an r.i.~Banach function space.
\end{corollary}

\begin{proof}
	This result follows from Theorem~\ref{TISS} and Theorem~\ref{TFA}.
\end{proof}

\begin{corollary} \label{CIASAS}
	Let $\lVert \cdot \rVert_X$ be an r.i.~quasi-Banach function norm and let $X$ be the corresponding r.i.~quasi-Banach function space. If $\lVert \cdot \rVert_X$ has the property \ref{P5} then $X_i' = X'$ up to equivalence of norms.
\end{corollary}

\begin{proof}
	This result follows directly from Theorem~\ref{TISS}.
\end{proof}

Next we obtain an analogue to Proposition~\ref{PEASG}.

\begin{corollary} \label{CEIAS}
	Let $\lVert \cdot \rVert_X$ and $\lVert \cdot \rVert_Y$ be a pair of r.i.~quasi-Banach function norms and let $X$ and $Y$, respectively, be the corresponding r.i.~quasi-Banach function spaces. Suppose that $X \hookrightarrow Y$. Then the respective integrable associate spaces $X_i'$ and $Y_i'$ satisfy $X_i' \hookleftarrow Y_i'$. 
\end{corollary}

\begin{proof}
	Thanks to our assumptions the respective integrable subspaces $X_i$ and $Y_i$ of $X$ and $Y$ satisfy $X_i \hookrightarrow Y_i$ and thus the result follows from Proposition~\ref{PEASG}.
\end{proof}

Finally, we now formulate the appropriate versions of Hölder inequality and Landau's resonance theorem. Note that while the latter can be formulated only in terms of the original quasinorm, the Hölder inequality requires us to work with the integrable norm which is rather unfortunate.

\begin{corollary}\label{THIAS}
	Let $\lVert \cdot \rVert_X$ be an r.i.~quasi-Banach function norm and let $\lVert \cdot \rVert_{X_i}$ and $\lVert \cdot \rVert_{X_i'}$, respectively, be the corresponding integrable norm and integrable associate norm. Then it holds for every pair of functions $f, g \in M$	that
	\begin{equation*}
	\int_R \lvert fg \rvert \: d\mu \leq \lVert f \rVert_{X_i'} \lVert g \rVert_{X_i}.
	\end{equation*}
\end{corollary}

\begin{corollary} \label{GLT}
	Let $\lVert \cdot \rVert_X$ be an r.i.~quasi-Banach function norm, let $X$ be the corresponding r.i.~quasi-Banach function space and let $\lVert \cdot \rVert_{X_i'}$ and $X_i'$, respectively, be the corresponding integrable associate norm and integrable associate space. Then arbitrary function $f \in M$ belongs to $X_i'$ if and only if it satisfies
	\begin{equation} \label{GLT1}
	\int_R \lvert f g \rvert \: d\mu < \infty
	\end{equation}
	for all $g \in X$ such that
	\begin{equation} \label{GLT2}
	\int_0^{1} g^* \: d\lambda < \infty.
	\end{equation}
\end{corollary}

\begin{proof}
	This result is a consequence of Theorem~\ref{LT} and \eqref{TISS1}.
\end{proof}

\subsection{The second integrable associate space}

We now answer the natural question how does the second integrable associate space relate to the original r.i.~quasi-Banach function space. It turns out that the answer is much more interesting than in the case of associate spaces.

\begin{definition} \label{DSIAN}
	Let $\lVert \cdot \rVert_X$ be an r.i.~quasi-Banach function norm, let $X$ be the corresponding r.i.~quasi-Banach function space and let $\lVert \cdot \rVert_{X_i'}$ and $X_i'$, respectively, be the corresponding integrable associate norm and integrable associate space. We then define the second integrable associate norm $\lVert \cdot \rVert_{X_i''}$ as the integrable associate norm of $\lVert \cdot \rVert_{X_i'}$ and the second integrable associate space $X_i''$ as the integrable associate space of $X_i'$.
\end{definition}

It follows directly from Corollary~\ref{CIAS} and Corollary~\ref{CIASAS} that the second integrable associate norm is equivalent to the associate norm of $\lVert \cdot \rVert_{X_i'}$. An analogous claim holds for the second integrable associate space.

\begin{theorem} \label{TSIAS}
	Let $\lVert \cdot \rVert_X$ be an r.i.~quasi-Banach function norm and let $X$ be the corresponding r.i.~quasi-Banach function space. Then:
	\begin{enumerate}
		\item \label{TSIAS1}
		if $\lVert \cdot \rVert_X$ is a Banach function norm, then $X = X_i''$ with equivalent norms;
		\item \label{TSIAS2}
		if $\lVert \cdot \rVert_X$ has the property \ref{P5} but it is not equivalent to a Banach function norm, then $X \hookrightarrow X_i''$ and $X_i'' \not \hookrightarrow X$;
		\item \label{TSIAS3}
		if $\lVert \cdot \rVert_X$ has the property \ref{P1} but not the property \ref{P5}, then $X_i'' \hookrightarrow X$ and $X \not \hookrightarrow X_i''$;
		\item \label{TSIAS4}
		if $\lVert \cdot \rVert_X$ has neither the property \ref{P1} nor the property \ref{P5}, then the spaces $X$ and $X_i''$ cannot, in general, be compared.
	\end{enumerate}

\end{theorem}

\begin{proof}
	The assertion \ref{TSIAS1} follows immediately from Corollary~\ref{CIASAS} and Theorem~\ref{TDAS}.
	
	The positive part of assertion \ref{TSIAS2} follows from Corollary~\ref{CIASAS} and Proposition~\ref{PESSAS} while the negative part follows from the fact that $X_i''$ is a Banach function space by Corollary~\ref{CIAS}.
	
	In the case \ref{TSIAS3} we know from Theorem~\ref{TISS} that the integrable subspace $X_i$ is a Banach function space, hence it follows from Corollary~\ref{CIASAS}, and Theorem~\ref{TDAS} that 
	\begin{equation*}
	X_i'' = (X_i)'' = X_i \hookrightarrow X.
	\end{equation*}
	On the other hand, since $X_i''$ is a Banach function space by Corollary~\ref{CIAS}, we get from Theorem~\ref{TP5} that there is a function in $X$ that does not belong to $X_i''$.
	
	Finally, to observe that there needs not to be any embedding in the case \ref{TSIAS4} one only has to consider the Lebesgue space $L^p$ with $p<1$. Indeed, in this case the space $(L^p)_i''$ is a Banach function space and therefore we have from Theorem~\ref{TLL} that
	\begin{equation*}
	WL(L^{\infty}, L^1) \hookrightarrow (L^p)_i'' \hookrightarrow WL(L^1, L^{\infty}),
	\end{equation*}
	while there are some functions in $WL(L^{\infty}, L^1)$ that do not belong to $L^p$ and at the same time there are some functions in $L^p$ that do not belong to $WL(L^1, L^{\infty})$.
\end{proof}

\section{Wiener--Luxemburg amalgams of quasi-Banach function spaces} \label{CHWLASq}

In this section we extend the theory of Wiener--Luxemburg amalgams to the context of r.i.~quasi-Banach function spaces. This is possible thanks to the recent advances of the general theory of quasi-Banach function spaces developed in \cite{NekvindaPesa20}. We focus on those areas where the generalisation leads to new results and insights.

Throughout this section we restrict ourselves to the case when $(R, \mu) = ([0, \infty), \lambda)$, which ensures that the underlying measure space is resonant. This restriction is necessary because we need to work with the non-increasing rearrangement and, in contrast to the situation in Section~\ref{CHWLAS}, the representation theory is not available for r.i.~quasi Banach function spaces.

\subsection{Wiener--Luxemburg quasinorms for quasi-Banach function spaces}

The definition and the basic properties of Wiener--Luxemburg quasinorms are subject only to the most natural and expected changes.
 
\begin{definition} \label{DefWLq}
	Let $\lVert \cdot \rVert_A$ and $\lVert \cdot \rVert_B$ be r.i.~quasi-Banach function norms. We then define the Wiener--Luxemburg quasinorm $\lVert \cdot \rVert_{WL(A, B)}$, for $f \in M$, by
	\begin{equation}
	\lVert f \rVert_{WL(A, B)} = \lVert f^* \chi_{[0,1]} \rVert_A + \lVert f^* \chi_{(1, \infty)} \rVert_B \label{DefWLNq}
	\end{equation}
	and the corresponding Wiener--Luxemburg amalgam space $WL(A, B)$ as
	\begin{equation*}
	WL(A, B) = \{f \in M; \; \lVert f \rVert_{WL(A, B)} < \infty \}.
	\end{equation*}
	
	Furthermore, we will call the first summand in \eqref{DefWLNq} the local component of $\lVert \cdot \rVert_{WL(A, B)}$ while the second summand will be called the global component of $\lVert \cdot \rVert_{WL(A, B)}$.
\end{definition}

\begin{remark}
	Let $\lVert \cdot \rVert_A$ be an r.i.~quasi-Banach function norm and denote by $C$ its modulus of concavity. Then
	\begin{equation*}
	\lVert f \rVert_A \leq \lVert f \rVert_{WL(A, A)} \leq 2C \lVert f \rVert_A
	\end{equation*}
	for every $f \in M$.
	
	Consequently, it makes a good sense to talk about local and global components of arbitrary r.i.~quasi-Banach function norms. 
\end{remark}

\begin{theorem} \label{TQNq}
	Let $\lVert \cdot \rVert_A$ and $\lVert \cdot \rVert_B$ be r.i.~quasi-Banach function norms. Then the Wiener--Luxemburg quasinorm $\lVert \cdot \rVert_{WL(A, B)}$ is an r.i.~quasi-Banach function norm. Consequently, the corresponding Wiener--Luxemburg amalgam $WL(A,B)$ is a rearrangement-invariant quasi-Banach function space.
	
	Moreover, $\lVert \cdot \rVert_{WL(A, B)}$ has the property $\ref{P5}$ if and only if $\lVert \cdot \rVert_A$ does.
\end{theorem}

\begin{proof}
	The proof of the first part, that is that $\lVert \cdot \rVert_{WL(A, B)}$ has all the properties of an r.i.~quasi-Banach function norm, is analogous to that of Theorem~\ref{TQN}.
	
	The proof that if $\lVert \cdot \rVert_A$ has the property $\ref{P5}$ then the same is true for $\lVert \cdot \rVert_{WL(A, B)}$ is analogous to the appropriate part of the proof of Proposition~\ref{PLC}.
	
	Consider now the case when $\lVert \cdot \rVert_A$ does not have the property $\ref{P5}$ and fix some $E \subseteq [0, \infty)$ that serves as an appropriate counterexample. It follows from Theorem~\ref{TP5} that there is a non-negative function $f_E \in A$ such that
	\begin{equation*}
		\infty = \int_E f_E \: d\lambda \leq \int_0^{\lambda(E)} f_E^* \: d\lambda,
	\end{equation*}
	where the last estimate follows from the Hardy--Littlewood inequality (Theorem~\ref{THLI}). The function $f = f_E^* \chi_{[0,1]}$ then satisfies $\lVert f \rVert_{WL(A,B)} < \infty$ while
	\begin{equation*}
		\int_0^{1} f \: d\lambda = \infty.
	\end{equation*}
\end{proof}

\subsection{Integrable associate spaces of Wiener--Luxemburg amalgams}

An interesting question is how to describe the associate space of a Wiener--Luxemburg amalgam $WL(A,B)$ in the case when $A$ has the property \ref{P5} but $B$ does not. It follows from Theorems~\ref{TFA} and \ref{TQNq} that it will be a Banach function space, but it cannot be described as $WL(A', B')$ since in this case $B' = \{0\}$. Trying to answer this question is what motivates the introduction of integrable associate spaces in the previous section.

The answer to this question will follow as a corollary to the following general theorem, in which we describe the integrable associate spaces of Wiener--Luxemburg amalgams.

\begin{theorem} \label{TASq}
	Let $\lVert \cdot \rVert_A$ and $\lVert \cdot \rVert_B$ be r.i.~quasi-Banach function norms and let $\lVert \cdot \rVert_{A_i'}$ and $\lVert \cdot \rVert_{B_i'}$ be their respective integrable associate norms. Then there is a constant $C>0$ such that the integrable associate norm $\lVert \cdot \rVert_{(WL(A, B))_i'}$ of $\lVert \cdot \rVert_{WL(A, B)}$ satisfies
	\begin{equation}\label{TASq1}
		\lVert f \rVert_{(WL(A, B))_i'} \leq \lVert f \rVert_{WL(A_i', B_i')} \leq C \lVert f \rVert_{(WL(A, B))_i'}
	\end{equation}
	for every $f \in M$.
	
	Consequently, the corresponding integrable associate space satisfies
	\begin{equation*}
		(WL(A, B))_i' = WL(A_i', B_i'),
	\end{equation*}
	up to equivalence of defining functionals.
\end{theorem}

\begin{proof}
	We begin by showing the first inequality in \eqref{TASq1}. To this end, fix some $f \in M$ and arbitrary $g \in M$ satisfying $\lVert g \rVert_{WL(A, B)_i} < \infty$. It then follows from Corollary~\ref{THIAS} that
	\begin{equation*}
	\begin{split}
		\int_0^{\infty} f^* g^* \: d\lambda &= \int_0^{\infty} f^*\chi_{[0,1]} g^* \: d\lambda + \int_0^{\infty} f^* \chi_{(1, \infty)} g^* \: d\lambda \\
		&\leq  \lVert f^* \chi_{[0,1]} \rVert_{A_i'} \lVert g^* \chi_{[0,1]} \rVert_{A_i} + \lVert f^* \chi_{(1,\infty)} \rVert_{B_i'} \lVert g^* \chi_{(1,\infty)} \rVert_{B_i} \\
		&\leq \lVert f \rVert_{WL(A_i', B_i')} \cdot \max \{\lVert g^* \chi_{[0,1]} \rVert_{A_i}, \, \lVert g^* \chi_{(1,\infty)} \rVert_{B_i} \} \\
		&\leq \lVert f \rVert_{WL(A_i', B_i')} \lVert g \rVert_{WL(A, B)_i}.
	\end{split}
	\end{equation*}
	The desired inequality now follows by dividing both sides by $\lVert g \rVert_{WL(A, B)_i}$, taking the appropriate supremum and using Proposition~\ref{PAS}.
	
	The second inequality in \eqref{TASq1} is more involved. We obtain it indirectly, showing first that $(WL(A,B))' \subseteq WL(A', B')$ and then using Theorem~\ref{TEQBFS}.
	
	Suppose that $f \notin WL(A_i', B_i')$. Then $f^* \chi_{[0,1]} \notin A_i'$ or $f^* \chi_{(1,\infty)} \notin B_i'$. We treat these two cases separately.
	
	If $f^* \chi_{[0,1]} \notin A_i'$ then we get by Corollary~\ref{GLT} that there is a non-negative function $g \in A$ such that
	\begin{align*}
		\int_0^{1} g^* \: d\lambda &< \infty, \\
		\int_0^{\infty} f^* \chi_{[0,1]} g \: d\lambda &= \infty.
	\end{align*}
	Now, $g^* \chi_{[0,1]} \in WL(A,B)$ because
	\begin{equation*}
		\lVert g^* \chi_{[0,1]} \rVert_{WL(A,B)} = \lVert g^* \chi_{[0,1]} \rVert_A \leq \lVert g^* \rVert_A = \lVert g \rVert_A < \infty.
	\end{equation*}
	Moreover, it obviously holds that
	\begin{equation*}
		\int_0^{1} (g^*\chi_{[0,1]})^* \: d\lambda = \int_0^{1} g^* \: d\lambda < \infty
	\end{equation*}
	and we get, by the Hardy--Littlewood inequality (Theorem~\ref{THLI}), the following estimate:
	\begin{equation*}
		\infty = \int_0^{\infty} f^* \chi_{[0,1]} g \: d\lambda \leq \int_0^{\infty} f^* g^* \chi_{[0,1]} \: d\lambda.
	\end{equation*}
	It thus follows from Corollary~\ref{GLT} that $f \notin (WL(A,B))_i'$.
	
	Suppose now that $f^* \chi_{(1,\infty)} \notin B_i'$. We may assume that $f^*(1) < \infty$, because otherwise $f^* \chi_{[0,1]} = \infty \chi_{[0,1]} \notin A_i'$ (see \cite[Lemma~2.4]{MizutaNekvinda15} or \cite[Theorem~3.4]{NekvindaPesa20}) and thus $f \notin (WL(A,B))_i'$ by the argument above. As in the previous case, we get by Corollary~\ref{GLT} that there is some non-negative function $g \in B$ such that
	\begin{align}
		\int_0^{1} g^* \: d\lambda &< \infty, \label{TASq2} \\
		\int_0^{\infty} f^* \chi_{(1,\infty)} g \: d\lambda &= \infty. \nonumber
	\end{align}
	Now, it holds for all $t \in (0, \infty)$ that
	\begin{equation*}
		(f^* \chi_{(1,\infty)})^*(t) = f^*(t+1),
	\end{equation*}
	which, when combined with the Hardy--Littlewood inequality (Theorem~\ref{THLI}), yields
	\begin{equation*}
		\infty = \int_0^{\infty} f^* \chi_{(1,\infty)} g \: d\lambda \leq \int_0^{\infty} f^*(t+1) g^*(t) \: dt = \int_1^{\infty} f^*(t) g^*(t-1) \: dt.
	\end{equation*}
	If we now put
	\begin{equation*}
		\tilde{g}(t) = \begin{cases}
		0 & \text{for } t \in [0,1], \\
		g^*(t-1) & \text{for } t \in (1, \infty),
		\end{cases}
	\end{equation*}
	we immediately see that $\tilde{g}^* = g^*$ and thus we have found a function $\tilde{g} \in B$ that is zero on $[0, 1]$, non-increasing on $(1, \infty)$ and that satisfies
	\begin{equation*}
		\int_0^{\infty} f^* \chi_{(1,\infty)} \tilde{g} \: d\lambda = \infty.
	\end{equation*}
	Furthermore, we may estimate by \eqref{TASq2}
	\begin{equation*}
		\int_1^{2} f^* \tilde{g} \: d\lambda \leq f^*(1) \int_1^2 \tilde{g} \: d\lambda = f^*(1) \int_0^{1} g^* \: d\lambda  < \infty.
	\end{equation*}
	It follows that
	\begin{equation*}
		\int_2^{\infty} f^* \tilde{g} \: d\lambda = \infty,
	\end{equation*}
	because
	\begin{equation*}
		\infty = \int_0^{\infty} f^* \chi_{(1,\infty)} \tilde{g} \: d\lambda = \int_1^{2} f^* \tilde{g} \: d\lambda + \int_2^{\infty} f^* \tilde{g} \: d\lambda.
	\end{equation*}
	
	Finally, put $h = \tilde{g}(2) \chi_{[0,1]} + \min\{\tilde{g}, \tilde{g}(2)\}$. Note that $\tilde{g}(2) < \infty$ as follows from $\tilde{g} \in B$ (see again \cite[Lemma~2.4]{MizutaNekvinda15} or \cite[Theorem~3.4]{NekvindaPesa20}) and that $h$ is therefore a finite non-increasing function. Hence, we get that
	\begin{gather*}
		\lVert h \rVert_{WL(A,B)} = \tilde{g}(2) \lVert \chi_{(0,1)} \rVert_A + \lVert \min\{\tilde{g}, \tilde{g}(2)\} \rVert_{B} \leq \tilde{g}(2) \lVert \chi_{(0,1)} \rVert_A + \lVert \tilde{g} \rVert_{B} < \infty, \\
		\int_0^{1} h^* \: d\lambda = \tilde{g}(2) < \infty,
	\end{gather*}
	while by the arguments above we have
	\begin{equation*}
		\int_0^{\infty} f^* h^* \: d\lambda \geq \int_2^{\infty} f^* h^* \: d\lambda = \int_2^{\infty} f^* \tilde{g} \: d\lambda = \infty.
	\end{equation*}
	We therefore get from Corollary~\ref{GLT} that $f \notin (WL(A,B))_i'$. This covers the last case and establishes the desired inclusion $(WL(A,B))_i' \subseteq WL(A_i', B_i')$.
	
	Because we already know from Theorem~\ref{TQNq} that $WL(A_i', B_i')$ is a quasi-Banach function space and from Theorem~\ref{CIAS} that $(WL(A,B))_i'$ is a Banach function space, we may use Theorem~\ref{TEQBFS} to obtain $(WL(A,B))_i' \hookrightarrow WL(A_i', B_i')$, i.e.~that there is a constant $C_2>0$ such that it holds for all $f \in M$ that
	\begin{equation*}
		\lVert f \rVert_{WL(A_i', B_i')} \leq C_2 \lVert f \rVert_{(WL(A, B))_i'},
	\end{equation*}
	which concludes the proof.
\end{proof}

\begin{corollary}
	Let $\lVert \cdot \rVert_A$ be an r.i.~quasi-Banach function norm that has the property \ref{P5}, let $\lVert \cdot \rVert_B$ be an r.i.~quasi-Banach function norm, let $\lVert \cdot \rVert_{A'}$ be the associate norm of $\lVert \cdot \rVert_A$ and  let $\lVert \cdot \rVert_{B_i'}$ be the integrable associate norm of $\lVert \cdot \rVert_B$. Then there is a constant $C>0$ such that the associate norm $\lVert \cdot \rVert_{(WL(A, B))'}$ of $\lVert \cdot \rVert_{WL(A, B)}$ satisfies
	\begin{equation*}
		\lVert f \rVert_{(WL(A, B))'} \leq \lVert f \rVert_{WL(A', B_i')} \leq C \lVert f \rVert_{(WL(A, B))'}
	\end{equation*}
	for every $f \in M$.
	
	Consequently, the corresponding associate space satisfies
	\begin{equation*}
		(WL(A, B))' = WL(A', B_i'),
	\end{equation*}
	up to equivalence of defining functionals.
\end{corollary}

\begin{proof}
	The result follows by combining Theorems~\ref{TQNq} and \ref{TASq} and Corollary~\ref{CIASAS}.
\end{proof}

\subsection{Embeddings}

The characterisation of embeddings remains the same and this is also true for its proof (which we thus omit). We state it here only because we will use it later.

\begin{theorem} \label{TEMq}
	Let $\lVert \cdot \rVert_A$, $\lVert \cdot \rVert_B$, $\lVert \cdot \rVert_C$ and $\lVert \cdot \rVert_D$ be r.i.~quasi-Banach function norms. Then the following assertions are true:
	\begin{enumerate}
		\item The embedding $WL(A, C) \hookrightarrow WL(B,C)$ holds if and only if the local component of $\lVert \cdot \rVert_A$ is stronger than that of $\lVert \cdot \rVert_B$, in the sense that for every $f \in M$ the implication
		\begin{equation*}
		\lVert f^* \chi_{[0,1]} \rVert_A < \infty \Rightarrow \lVert f^* \chi_{[0,1]} \rVert_B < \infty
		\end{equation*} \label{TEMqp1}
		holds.
		\item The embedding $WL(A, B) \hookrightarrow WL(A,C)$ holds if and only if the global component of $\lVert \cdot \rVert_B$ is stronger than that of $\lVert \cdot \rVert_C$, in the sense that for every $f \in M$ such that $f^*(1) < \infty$ the implication
		\begin{equation*}
		\lVert f^* \chi_{(1, \infty)} \rVert_B < \infty \Rightarrow \lVert f^* \chi_{(1, \infty)} \rVert_C < \infty
		\end{equation*}
		holds. \label{TEMqp2}
		\item The embedding $WL(A,B) \hookrightarrow WL(C,D)$ holds if and only if the local component of $\lVert \cdot \rVert_A$ is stronger than that of $\lVert \cdot \rVert_C$ and the global component of $\lVert \cdot \rVert_B$ is stronger than that of $\lVert \cdot \rVert_D$. \label{TEMqp3}
	\end{enumerate}
\end{theorem}

The following theorem generalises Theorem~\ref{TLL} and also provides better insight into the relationships between the individual embeddings and the specific properties a quasi-Banach function norm can possess. We formulate it in a simpler way than Theorem~\ref{TLL}, but this comes at no loss of generality thanks to Theorem~\ref{TEMq}.

\begin{theorem} \label{TLLq}
	Let $\lVert \cdot \rVert_A$ be an r.i.~quasi-Banach function norm. Then
	\begin{enumerate}
		\item  $WL(L^{\infty}, A) \hookrightarrow A$, \label{TLLqp1}
		\item  if $\lVert \cdot \rVert_A$ has the property \ref{P1} then $WL(A, L^1) \hookrightarrow A$, \label{TLLqp2}
		\item  $A \hookrightarrow WL(L^1, A)$ if and only if $\lVert \cdot \rVert_A$ has the property \ref{P5}, \label{TLLqp3}
		\item  $A \hookrightarrow WL(A, L^{\infty})$. \label{TLLqp4}
	\end{enumerate}
\end{theorem}

\begin{proof}
	The first and last embeddings are proved in the same way as in Theorem~\ref{TLL}.
	
	The sufficiency in part \ref{TLLqp3} follows exactly as in the part \ref{TLLp3} of Theorem~\ref{TLL}. For the necessity, assume that $A \hookrightarrow WL(L^1, A)$. It then holds for any $E \subseteq [0, \infty)$ of finite measure and any $f \in A$ that
	\begin{equation*}
		\int_E \lvert f \rvert \: d\lambda \leq \int_0^{\lambda(E)} f^* \: d\lambda \leq \max\{1, \, \lambda(E)\} \int_0^{1} f^* \: d\lambda < \infty,
	\end{equation*}
	where the first estimate is due to the Hardy--Littlewood inequality (Theorem~\ref{THLI}), the second estimate is due to $f^*$ being non-increasing, and the last estimate is due to part \ref{TEMqp1} of Theorem~\ref{TEMq}. We now obtain from Theorem~\ref{TP5} that $\lVert \cdot \rVert_A$ must have the property \ref{P5}.
	
	As for the part \ref{TLLqp2}, we know from Theorem~\ref{TISS} that if $\lVert \cdot \rVert_A$ has the property \ref{P1} then its integrable subspace $A_i$ is an r.i.~Banach function space. Hence, we obtain from parts \ref{TLLp1} and \ref{TLLp2} of Theorem~\ref{TLL} that
	\begin{equation*}
		WL(L^{\infty}, L^1) \hookrightarrow A_i \hookrightarrow A.
	\end{equation*}
	The desired conclusion now follows by combining parts \ref{TEMqp3} and \ref{TEMqp2} of Theorem~\ref{TEMq}.
\end{proof}

\begin{remark}
	Unlike in part \ref{TLLqp3} of the preceding theorem, there is no equivalence in part \ref{TLLqp2}. This can be observed by considering $A = L^{p,q}$, where $L^{p,q}$ is a Lorenz space, because if we choose $p \in (1, \infty)$, $q \in (0, 1)$ then $L^{p,q}$ satisfies the embedding (see Remark~\ref{RELpq}) but is not normable (see \cite[Theorem~2.5.8]{CarroRaposo07} and the references therein).
\end{remark}

An alternative sufficient condition for the embedding $WL(A, L^1) \hookrightarrow A$ is provided in the following theorem. The relevant term is defined in Definition~\ref{DHLP} and put into context in Lemma~\ref{LHLP} and Remark~\ref{RHLP}.

\begin{theorem} \label{THLP}
	Let $\lVert \cdot \rVert_A$ be an r.i.~quasi-Banach function norm and assume that the Hardy--Littlewood--P\'{o}lya principle holds for $\lVert \cdot \rVert_A$. Then the global component of $\lVert \cdot \rVert_{L^1}$ is stronger than that of $\lVert \cdot \rVert_A$.
\end{theorem}

\begin{proof}
	Fix some $f \in M$ such that $f^*(1) < \infty$ and $\lVert f^* \chi_{(1, \infty)} \rVert_{L^1} < \infty$. We want to show that $\lVert f^* \chi_{(1, \infty)} \rVert_{A} < \infty$. To this end, consider the non-increasing function $f_0 = f^*(1) \chi_{[0,1]} + f^* \chi_{(1, \infty)}$ which belongs to $L^1$ and the norm of which satisfies
	\begin{equation*}
		f^*(1) \leq \left \lVert f_0 \right \rVert_{L^1} < \infty.
	\end{equation*}
	Hence, we may define a function $h = \left \lVert f_0 \right \rVert_{L^1} \chi_{[0,1]}$ and observe that $h \in A$, $h$ is non-increasing, and
	\begin{equation*}
		\int_0^{t} f_0^* \: d\lambda \leq \int_0^{t} h^* \: d\lambda 
	\end{equation*}
	for all $t \in (0, \infty)$. It thus follows from our assumption on $\lVert \cdot \rVert_A$ that $f_0 \in A$ and consequently $\lVert f^* \chi_{(1, \infty)} \rVert_{A} < \infty$, as desired.
\end{proof}

As a corollary to this theorem we obtain a negative answer to the previously open question whether the Hardy--Littlewood--P\'{o}lya principle holds for every r.i.~quasi-Banach function norm. This also serves as an example of application of Wiener--Luxemburg amalgams, since they appear only in the proof of the corollary, not in its statement.

\begin{corollary} \label{CHLP}
	There is an r.i.~quasi-Banach function norm over $M((0, \infty), \lambda)$ which has the property \ref{P5} and for which the Hardy--Littlewood--P\'{o}lya principle does not hold.
\end{corollary}

\begin{proof}
	Let $p \in (0, 1)$. It follows from Theorem~\ref{TQNq}, Theorem~\ref{TEMq}, Remark~\ref{RELp}, and Theorem~\ref{THLP} that the spaces $WL(L^1, L^p)$ have the desired properties.
\end{proof}

Finally, we present a result that generalises Proposition~\ref{PEAS} and which is useful when one wants to find an integrable associate space to a given r.i.~quasi-Banach function space.

\begin{proposition} \label{PEIAS}
	Let $\lVert \cdot \rVert_A$ and $\lVert \cdot \rVert_B$ be r.i.~quasi-Banach function norms and denote by $\lVert \cdot \rVert_ {A_i'}$ and $\lVert \cdot \rVert_{B_i'}$ the respective integrable associate norms. Suppose that the local component of $\lVert \cdot \rVert_A$ is stronger than that of $\lVert \cdot \rVert_B$. Then the local component of $\lVert \cdot \rVert_{B_i'}$ is stronger than that of $\lVert \cdot \rVert_{A_i'}$.
	
	Similarly, if the global component of $\lVert \cdot \rVert_A$ is stronger than that of $\lVert \cdot \rVert_B$, then the global component of $\lVert \cdot \rVert_{B_i'}$ is stronger than that of $\lVert \cdot \rVert_{A_i'}$
\end{proposition}

\begin{proof}
	By our assumption and part~\ref{TEMqp1} of~Theorem~\ref{TEMq} we get that
	\begin{equation*}
	WL(A, L^{\infty}) \hookrightarrow WL(B, L^{\infty}).
	\end{equation*}
	Consequently, it follows from  Theorem~\ref{TASq}, Corollary~\ref{CIASAS} and Corollary~\ref{CEIAS} that
	\begin{equation*}
	WL(B_i', L^1) = (WL(B, L^{\infty}))_i' \hookrightarrow (WL(A, L^{\infty}))_i' = WL(A_i', L^1),
	\end{equation*}
	that is, the local component of $\lVert \cdot \rVert_{B_i'}$ is stronger than that of $\lVert \cdot \rVert_ {A_i'}$.
	
	The second claim can be proved in similar manner, only using $WL(L^1, A)$ and $WL(L^1, B)$ instead of $WL(A, L^{\infty})$ and $WL(B, L^{\infty})$.
\end{proof}

\appendix
\section{Some remarks on Wiener amalgams} \label{Appendix}

In the appendix we again restrict ourselves to the case when $(R, \mu) = ([0, \infty), \lambda)$ and consider Wiener amalgams of the classical Lebesgue spaces. We show that even this very simple setting is sufficient to obtain an example of a Wiener amalgam which is constructed from a pair of r.i.~Banach function spaces and which is neither rearrangement invariant nor a Banach function space. The reason for this behaviour is that Wiener amalgams treat locality and globality in the topological sense, while Banach function spaces do so in the measure-theoretic sense.

We start by providing the standard simple definition of a Wiener amalgam of Lebesgue spaces (see for example \cite{FournierSteward85} or \cite{Holland75}).

\begin{definition}
	Let $p,q \in [1,\infty]$. Consider the classical Lebesgue spaces $L^p$ of functions belonging to $M([0, \infty), \lambda)$ and $l^q$ of sequences belonging to $M(\mathbb{N}, m)$. We then define, for all $f \in M([0, \infty), \lambda)$, the Wiener norm $\lVert \cdot \rVert_{W(L^p, l^q)}$ by
	\begin{align*}
	\lVert f \rVert_{W(L^p, l^q)} &= \left ( \sum_{n = 0}^{\infty} \lVert f \chi_{[n, n+1)} \rVert_{L^p}^q \right )^{1/q} &\text{for } q \in [1,\infty), \\
	\lVert f \rVert_{W(L^p, l^q)} &= \sup_{n \in \mathbb{N}} \lVert f \chi_{[n, n+1)} \rVert_{L^p} &\text{for } q = \infty,
	\end{align*}
	and the corresponding Wiener amalgam space (or just Wiener amalgam) by
	\begin{equation*}
	W(L^p, l^q) = \{f \in M([0, \infty), \lambda); \; \lVert f \rVert_{W(L^p, l^q))} < \infty \}. \label{DefWS}
	\end{equation*}
\end{definition}

The next proposition shows what properties the Wiener amalgams of Lebesgue spaces do have. Note that the difference in behaviour is precisely that the measure-theoretic conditions on $E$ in \ref{P4} and \ref{P5} are replaced with topological ones.

\begin{proposition} \label{PPWA}
	Let $p,q \in [1,\infty]$. Then the Wiener norm $\lVert \cdot \rVert_{W(L^p, l^q)}$ is indeed a norm and it also satisfies the axioms \ref{P2} and \ref{P3} of Banach function norms together with weaker versions of axioms \ref{P4} and \ref{P5}, namely
	\begin{enumerate}[label=\textup{(P\arabic*')}]
		\setcounter{enumi}{3}
		\item \label{P4'} it holds for every bounded $E \subseteq [0, \infty)$ that $\lVert \chi_E \rVert_{W(L^p, l^q)} <\infty$,
		\item \label{P5'} it holds for every bounded $E \subseteq [0, \infty)$ that there is a constant $C_E < \infty$ satisfying
		\begin{equation*}
		\int_E \lvert f \rvert \: d\lambda \leq C_E \lVert f \rVert_{W(L^p, l^q)}
		\end{equation*}
		for every $f \in M([0, \infty), \lambda)$.
	\end{enumerate}
\end{proposition}

\begin{proof}
	Only the properties \ref{P4'} and \ref{P5'} in the second assertion require a proof. We will cover only the case $q \in [1, \infty)$ since the remaining case is easier.
	
	Fix a bounded set $E \subseteq [0, \infty)$. Then there is $n_0 \in \mathbb{N}$ such that $E \cap [n,n+1) = \emptyset$ for every $n \geq n_0$. Thus the assertion \ref{P4'} follows from the properties \ref{P1}, \ref{P2} and \ref{P4} of $\lVert \cdot \rVert_{L^p}$, since they imply that all the summands in the definition of $\lVert \chi_E \rVert_{W(L^p,l^q)}$ are finite and only finitely many of them are greater than zero.
	
	Similarly, the property \ref{P5'} follows from the property \ref{P5} of $\lVert \cdot \rVert_{L^p}$, since it allows us to estimate
	\begin{equation*}
	\begin{split}
	\int_E \lvert f \rvert \: d\lambda &= \sum_{n = 0}^{\infty} \int_E \lvert f \rvert \chi_{[n, n+1)} \: d\lambda \leq \sum_{n = 0}^{n_0} C_E \lVert f \chi_{[n, n+1)} \rVert_{L^p} \\
	&\leq C_0 C_E \left ( \sum_{n = 0}^{n_0} \lVert f \chi_{[n, n+1)} \rVert_{L^p}^q  \right )^{\frac{1}{q}} \leq C \lVert f \rVert_{W(L^p, l^q)},
	\end{split}
	\end{equation*}
	where $C_E$ is the constant from the property \ref{P5} of $\lVert \cdot \rVert_{L^p}$ for the set $E$ and $C_0$ is the constant from the equivalence of $\lVert \cdot \rVert_{l^1}$ and $\lVert \cdot \rVert_{l^q}$ norms on $\mathbb{R}^{n_0}$.
\end{proof}

We now show that the stronger requirement on $E$ in \ref{P4'} and \ref{P5'} was necessary, i.e.~that the Wiener amalgams of Lebesgue spaces do not in general have the properties \ref{P4} and \ref{P5}.

\begin{proposition} \label{RWNBFS}
	~
	
	\begin{enumerate}
		\item Let $1 \leq q < p \leq \infty$. Then the norm $\lVert \cdot \rVert_{W(L^p, l^q)}$ does not satisfy \ref{P4}. \label{RWNBFS1}
		\item Let $1 \leq p < q \leq \infty$. Then the norm $\lVert \cdot \rVert_{W(L^p, l^q)}$ does not satisfy \ref{P5}. \label{RWNBFS2}
	\end{enumerate}
	
	Consequently, the norm $\lVert \cdot \rVert_{W(L^p, l^q)}$ is a Banach function norm if and only if $1 \leq p=q \leq \infty$, in which case it coincides with the classical Lebesgue norm $\lVert \cdot \rVert_p$.
\end{proposition}

\begin{proof}
	We will show part~\ref{RWNBFS1} only for $p < \infty$ since the remaining case is easier. Fix arbitrary $a \in \left (1, \frac{p}{q} \right )$ and define 
	\begin{equation*}
	E = \bigcup_{n \in \mathbb{N}} \left [ n, n+\frac{1}{n^a} \right ].
	\end{equation*}
	Then, by our assumptions on $a$,
	\begin{equation*}
	\lambda(E) = \sum_{n = 0}^{\infty} n^{-a} < \infty 
	\end{equation*}
	but
	\begin{equation*}
	\lVert \chi_E \rVert_{W(L^p, l^q)}^q = \sum_{n = 0}^{\infty} n^{-a \frac{q}{p}} = \infty. 
	\end{equation*}
	
	As for part~\ref{RWNBFS2}, we will show it only for $q < \infty$ since the remaining case is easier. Fix arbitrary $a \in \left ( \frac{p}{q}, 1 \right )$, $b \in \left ( 1, \frac{p-a}{p-1} \right)$ (if $p = 1$ then any $b \in (1, \infty)$ will suffice) and define 
	\begin{align*}
	E &= \bigcup_{n \in \mathbb{N}} \left [ n, n+\frac{1}{n^b} \right ], \\
	f &= \sum_{n=0}^{\infty} n^{\frac{b-a}{p}} \chi_{[n, n+n^{-b}]}.
	\end{align*}
	Then, by our assumptions on $a$ and $b$,
	\begin{align*}
	\lambda(E) &= \sum_{n = 0}^{\infty} n^{-b} < \infty, \\
	\lVert f \rVert_{W(L^p, l^q)}^q &= \sum_{n=0}^{\infty} \left ( \int_n^{n+1} \lvert f \rvert^p \: d\lambda  \right )^{\frac{q}{p}} = \sum_{n=0}^{\infty} n^{-a \frac{q}{p}} < \infty
	\end{align*}
	but
	\begin{equation*}
	\int_E \lvert f \rvert \: d\lambda = \sum_{n=0}^{\infty} n^{\frac{b-a}{p} - b} = \infty.
	\end{equation*}
\end{proof}

As for the rearrangement invariance, we consider the functional $f \mapsto \lVert f^* \rVert_{W(L^p, l^q)}$. It is obvious that $\lVert \cdot \rVert_{W(L^p, l^q)}$ can be equivalent to a rearrangement-invariant norm only if it is equivalent to this functional. But as we show in the following theorem, this functional is in fact equivalent to the Wiener--Luxemburg amalgam quasinorm $\lVert \cdot \rVert_{WL(L^p, L^q)}$, which is always equivalent to a Banach function norm by Corollary~\ref{CN}, and thus by the previous remark the required equivalence can hold only if $p=q$.

\begin{theorem} \label{TIA}
	Let $p,q \in [1,\infty]$. Then the functional
	\begin{equation}
		f \mapsto \lVert f^* \rVert_{W(L^p, l^q)} \label{TIA1}
	\end{equation}
	is an r.i.~quasi-Banach function norm. Furthermore, there is a constant $C \in (0, \infty)$ such that
	\begin{equation} \label{TIA2}
		C^{-1} \lVert f^* \rVert_{W(L^p, l^q)} \leq \lVert f \rVert_{WL(L^p, L^q)} \leq C \lVert f^* \rVert_{W(L^p, l^q)},
	\end{equation}
	for every $f \in M([0, \infty), \lambda)$.
\end{theorem}

\begin{proof}
	It is obvious that the functional defined by \eqref{TIA1} has the properties \ref{P2} and \ref{P3} and that it is rearrangement-invariant. Furthermore, if we have $E \subseteq [0, \infty)$ such that $\lambda(E) < \infty$ then $\chi_E^* = \chi_{[0, \lambda(E))}$ and we have by the Hardy--Littlewood inequality (Theorem~\ref{THLI}) that
	\begin{equation*}
		\int_E \lvert f \rvert \: d\lambda \leq \int_{[0, \lambda(E))} f^* \: d\lambda.
	\end{equation*}
	The properties \ref{P4} and \ref{P5} therefore follow from Proposition~\ref{PPWA}.
	
	As for the property \ref{Q1}, the parts \ref{Q1a} and \ref{Q1b} are obvious while the part \ref{Q1c} deserves further comment. Let us fix some $f, g \in M$. We may use a similar argument as in Proposition~\ref{TQN} to obtain for any $n \in \mathbb{N}$ that
	\begin{align*}
		\lVert (f+g)^* \chi_{[n,n+1)} \rVert_{L^p} &\leq \lVert (D_{\frac{1}{2}} f^* + D_{\frac{1}{2}} g^*) \chi_{[n,n+1)} \rVert_{L^p} \\
		&\leq \lVert D_{\frac{1}{2}} f^* \chi_{[n,n+1)} \rVert_{L^p} + \lVert D_{\frac{1}{2}} g^* \chi_{[n,n+1)} \rVert_{L^p} \\
		&\leq 2^{\frac{1}{p}} \left ( \lVert f^* \chi_{[\frac{n}{2}, \frac{n+1}{2})} \rVert_{L^p} + \lVert g^* \chi_{[\frac{n}{2}, \frac{n+1}{2})} \rVert_{L^p} \right ).
	\end{align*}	
	From this we obtain the desired estimate
	\begin{equation*}
		\lVert (f+g)^* \rVert_{W(L^p, l^q)} \leq 2^{\frac{1}{p}+\frac{1}{q}} \left ( \lVert f^* \rVert_{W(L^p, l^q)} + \lVert g^* \rVert_{W(L^p, l^q)} \right ).
	\end{equation*}
	
	It remains to verify \eqref{TIA2}, which will follow from Theorem~\ref{TEQBFS} once we show that it holds for all $f \in M$ that $\lVert f^* \rVert_{W(L^p, l^q)} < \infty$ if and only if $\lVert f \rVert_{WL(L^p,L^q)} < \infty$.
	
	Suppose that $\lVert f^* \rVert_{W(L^p, l^q)} < \infty$. Then we immediately observe that $\lVert f^* \chi_{[0,1]} \rVert_{L^p} < \infty$. Furthermore,
	\begin{equation*}
		\lVert f^* \chi_{(1,\infty)} \rVert_{L^q} = \left ( \sum_{n=1}^{\infty} \int_n^{n+1} (f^*)^q \: d\lambda  \right)^{\frac{1}{q}} \leq \left ( \sum_{n=1}^{\infty} (f^*(n))^q  \right)^{\frac{1}{q}} \leq \left ( \sum_{n = 1}^{\infty} \lVert f^* \chi_{[n-1, n)} \rVert_{L^p}^q \right )^{1/q} < \infty,
	\end{equation*}
	and thus $\lVert f \rVert_{WL(L^p,L^q)} < \infty$.
	
	Conversely, suppose that $\lVert f \rVert_{WL(L^p,L^q)} < \infty$. Then
	\begin{equation*}
	\begin{split}
		\lVert f^* \rVert_{W(L^p, l^q)} &= \left ( \sum_{n = 0}^{\infty} \lVert f^* \chi_{[n, n+1)} \rVert_{L^p}^q \right )^{1/q} \\
		&\leq \left ( \lVert f \chi_{[0, 1)} \rVert_{L^p}^q + \sum_{n = 1}^{\infty} f^*(n)^q \right )^{1/q} \\
		&\leq \left ( 2\lVert f \chi_{[0, 1)} \rVert_{L^p}^q + \sum_{n=2}^{\infty} \int_{n-1}^{n} (f^*)^q \: d\lambda \right )^{1/q} \\
		&=  \left ( 2\lVert f \chi_{[0, 1)} \rVert_{L^p}^q + \int_{1}^{\infty} (f^*)^q \: d\lambda \right )^{1/q} \\
		&< \infty
	\end{split}
	\end{equation*}
	and thus $\lVert f^* \rVert_{W(L^p, l^q)} < \infty$.
\end{proof}

\begin{corollary} \label{CWNERI}
	Let $p, q \in [1, \infty]$. Then the Wiener amalgam norm $\lVert \cdot \rVert_{W(L^p, l^q)}$ is equivalent to a rearrangement-invariant norm if and only if $p = q$ in which case it is the classical Lebesgue norm $\lVert \cdot \rVert_p$.
\end{corollary}

\bibliographystyle{dabbrv}
\bibliography{bibliography}
\end{document}